\documentclass[reqno]{amsart}
\usepackage[colorlinks=true,linkcolor=MidnightBlue,citecolor=Sepia]{hyperref}
\usepackage{graphics}
\usepackage{nicefrac}
\usepackage[english]{babel}
\usepackage[utf8]{inputenc}
\usepackage{graphicx}
\usepackage[format=plain, indention=1em, small, width=\textwidth]{caption}
\usepackage{calc}
\usepackage{verbatim}
\usepackage{float}

\restylefloat{figure}

\pdfoutput=1
\usepackage[dvipsnames]{xcolor}

\usepackage{amsmath}
\usepackage{amsfonts}
\usepackage{amssymb}
\usepackage{amstext}
\usepackage[all,line]{xy}
\usepackage{mathrsfs}
\usepackage{wasysym}
\usepackage{stmaryrd}
\usepackage{yhmath}
\usepackage{extarrows}
\usepackage{chemarrow}
\usepackage{subfigure}
\usepackage{authblk}

\usepackage{calligra}
\usepackage{clock}
\usepackage{scrtime}
\usepackage{scrdate}
\usepackage{swrule}
\usepackage{type1cm}    
\usepackage{lettrine}
\usepackage{pifont}
\usepackage{skull}
\usepackage[normalem]{ulem}

\usepackage{algorithm}
\usepackage[end]{algpseudocode}

\usepackage[marginpar]{todo}
\usepackage{enumerate}

\newcommand{\round}[1]{\ensuremath{\left(#1\right)}}

\newcommand{\setdist}[1]{\ensuremath{\text{\textnormal{dist}}\round{#1}}}

\newcommand{\R}{\mathbb{R}}

\newcommand{\N}{\mathbb{N}}
\newcommand{\B}{Y}

\newcommand{\cW}{{\mathcal{W}}}
\newcommand{\cB}{{\mathcal{B}}}
\newcommand{\cC}{{\mathcal{C}}}
\newcommand{\cP}{{\mathcal{P}}}

\newcommand{\cL}{\mathcal{L}}
\newcommand{\cA}{\mathcal{A}}
\newcommand{\cU}{\mathcal{U}}

\newcommand{\cN}{\mathcal{N}}

\def\diam{\mathop{\rm diam}\nolimits}

\newtheorem{theorem}{Theorem}[section]

\newtheorem{proposition}{Proposition}

\newtheorem{remark}{Remark}
\theoremstyle{definition}
\newtheorem{definition}[theorem]{Definition}

\usepackage[capitalize]{cleveref}

\title[]{The numerical computation of unstable manifolds for infinite dimensional dynamical systems by embedding techniques}
\author[A. Ziessler, M. Dellnitz, R. Gerlach]{Adrian Ziessler, Michael Dellnitz, Raphael Gerlach}

\email{ziessler@math.uni-paderborn.de}
\email{dellnitz@uni-paderborn.de}
\email{rgerlach@math.uni-paderborn.de}

\date{}

\begin{document}

\maketitle

\centerline{\scshape Adrian Ziessler, Michael Dellnitz and Raphael Gerlach}
\medskip
{\footnotesize
\centerline{Department of Mathematics}
\centerline{Paderborn University}
\centerline{33095 Paderborn, Germany}
}

\begin{abstract}
In this work we extend the novel framework developed in \cite{DHZ16} to the computation of finite dimensional unstable manifolds of infinite dimensional dynamical systems. To this end, we adapt a set-oriented continuation technique developed in \cite{DH96} for the computation of such objects of finite dimensional systems with the results obtained in \cite{DHZ16}. We show how to implement this approach for the analysis of partial differential equations and illustrate its feasibility by computing unstable manifolds of the one-dimensional Kuramoto-Sivashinsky equation as well as for the Mackey-Glass delay differential equation.
\end{abstract}

\section{Introduction}
Set-oriented numerical methods have been developed in the context of the numerical treatment of dynamical systems (e.\ g.\ \cite{DH97, DJ99, FD03, FLS10}). The basic idea of these methods is to cover the objects of interest -- for instance \emph{invariant sets} or \emph{global attractors} -- by outer approximations which are created via multilevel subdivision or continuation techniques. They have been used in several different application areas such as molecular dynamics (\cite{SHD01}), astrodynamics (\cite{DJLMPPRT05}) or ocean dynamics (\cite{FHRSSG12}).

The computation of (global) invariant manifolds has attracted a lot of interest in recent years. Approaches based on geometric concepts can be used to approximate (up to two-dimensional) unstable manifolds of vector fields (see \cite{KOD+05} for an overview). The approximation by geodesic level sets, for instance, produces a regular mesh that consists of geodesic circles by solving appropriate boundary value problems (e.g., \cite{KO03,KO07}). Topological methods are used to create outer approximations of invariant manifolds. In \cite{KOD+05} a survey of different approaches for computing global stable or unstable manifolds of vector fields is given. 

In this paper, we will focus on the set-oriented continuation method that has been developed in \cite{DH96} and we will show how to use it for the computation of unstable manifolds of infinite dimensional dynamical systems. We will, in particular, approximate unstable manifolds for semiflows of Banach spaces (cf.~\cite{Ha71,He06,Ca12,CH12}). Our approach relies on the results obtained in \cite{DHZ16}, where the subdivision technique developed in \cite{DH97} has been extended to the infinite dimensional context. The underlying idea is to compute low dimensional invariant sets of infinite dimensional dynamical systems by utilizing embedding techniques for infinite dimensional systems \cite{HuntKaloshin99,R05}. To this end, a delay embedding technique has been used in order to obtain a one-to-one representation of the (infinite dimensional) dynamics by the so-called \emph{core dynamical system (CDS)}. The CDS is a continuous dynamical system on a state space of moderate dimension which we will refer to as the \emph{observation space}. Then the subdivision scheme from \cite{DH97} has been applied to the CDS in order to analyze one-to-one images of global attractors for the infinite dimensional dynamical system. The feasibility of this approach has been illustrated by computing invariant sets for delay differential equations with constant time delay. Recently, this approach has also been generalized to delay differential equations with state dependent time delay \cite{AZ18}. In \cite{BKK+18} low-dimensional transition manifolds of stochastic dynamical systems showing metastable behavior have been parameterized by using a combination of embedding and manifold learning techniques.

In addition to delay differential equations, other application scenarios in which finite dimensional invariant sets arise are certain types of dissipative dynamical systems described by partial differential equations, for instance the Kuramoto-Sivashinsky equation \cite{KT76, Siv77}, the Ginzburg-Landau equation, or scalar reaction-diffusion equations with cubic nonlinearity \cite{Jolly89}. For all these systems a finite dimensional so-called \emph{inertial manifold} exists to which trajectories are attracted exponentially fast, e.g., \cite{CFNT88,FJK88,Temam97}.

In this paper we extend the classical set-oriented continuation technique developed in \cite{DH96} to the infinite dimensional context. This method allows us to approximate unstable manifolds of infinite dimensional dynamical systems in observation space. The general numerical approach we are proposing is in principle applicable to infinite dimensional dynamical systems described by a Lipschitz continuous operator on a Banach space. However, in this article we will focus on \emph{partial differential equations (PDEs)}. To this end, we will develop an appropriate numerical realization of the CDS for PDEs and illustrate the efficiency of our method for a one-dimensional Kuramoto-Sivashinsky equation. The Kuramoto-Sivashinsky equation has attracted a lot of interest as a model for complex spatio-temporal dynamics and has been derived in the context of several extended physical problems, e.g., phase dynamics in reaction-diffusion systems \cite{KT76} or small thermal diffusive instabilities in laminar flame fronts \cite{Siv77}. It is also a paradigm for the existence of complex finite dimensional dynamics in a PDE. In addition, we will also illustrate the efficiency of our method for the Mackey-Glass delay differential equation by using the numerical realization of the CDS introduced in \cite{DHZ16}. In the context of delay differential equations a related -- though not set-oriented -- approach has recently been utilized in \cite{GJ17} for the numerical approximation of unstable manifolds. In this work the authors compute high order Taylor and Fourier-Taylor approximations of unstable manifolds for equilibria or periodic solutions.

A detailed outline of the paper is as follows. In \Cref{sec:review} we briefly summarize the results of \cite{DHZ16}. We state the main results of \cite{HuntKaloshin99} and \cite{R05} and we describe the construction of the CDS on the observation space. In \Cref{sec:continuation} we first briefly review the continuation method developed in \cite{DH96} for the computation of unstable manifolds of finite dimensional dynamical systems. Then we explain how to use this technique in the context of infinite dimensional dynamical systems. In \cite{DH96} convergence has been shown for compact subsets of unstable manifolds. Here we will extend this convergence result to the closure of the entire unstable manifold. In particular, we will prove that in this setting the embedding of the local unstable manifold (in infinite dimensional space) is identical to the local embedded unstable manifold (cf.~\cref{prop:elle}~(b)). 
A numerical realization for the construction of the CDS for PDEs is given in \Cref{sec:num_real_PDE}.
Finally, in \Cref{sec:numex}, we illustrate the efficiency of our method for the Mackey-Glass delay differential equation and for a one-dimensional Kuramoto-Sivashinsky equation. We will, in particular, generate numerical approximations of one-to-one copies of unstable manifolds for the Kuramoto-Sivashinsky equation for different parameter values.

\section{Review of subdivision and embedding results}\label{sec:review}
Since our set-oriented continuation method is based on the framework developed in \cite{DHZ16} we start with a short review of
the related material.
\subsection{Basic definitions and results from embedding theory}\label{ssec:embedding_theory}
We consider dynamical systems of the form
\begin{equation}\label{eq:DS}
u_{j+1} = \Phi (u_j),\quad j=0,1,\ldots,
\end{equation}
where $\Phi:\B\rightarrow \B$ is Lipschitz continuous on a Banach space $\B$.
Moreover we assume that $\Phi$ has an invariant compact set $\cA$, that is
\[
\Phi(\cA) = \cA.
\]
This assumption is justified by several classical results, where it has been shown that for many dissipative systems on Banach spaces there exist (non-trivial) compact attractors of finite capacity or Hausdorff dimension (see \cite{MP76,Ma81,CFT85,Ch88,Ha10}).
In order to approximate the set $\cA$ we combine a classical subdivision technique for the computation of such objects in a finite dimensional space with infinite dimensional embedding results (cf. \cite{HuntKaloshin99,R05}). 
For the statement of the main result of \cite{R05} we need three
particular notions:
\emph{prevalence} \cite{SYC91}, \emph{upper box counting dimension} and \emph{thickness
	exponent} \cite{HuntKaloshin99}.

\begin{definition}\quad
	\begin{enumerate}
		\item[(a)] A Borel subset $S$ of a normed linear space $V$ is
		\emph{prevalent} if there is a finite dimensional subspace $E$ of
		$V$ (the `probe space') such that for each $v \in V,\ v+e$ belongs
		to $S$ for (Lebesgue) almost every $e\in E$.
		\item[(b)] Let $\B$ be a Banach space, and let $\cA\subset \B$ be compact.
		For $\varepsilon >0$,
		denote by $N_\B(\cA,\varepsilon)$ the minimal number of balls of radius
		$\varepsilon$ (in the norm of $\B$) necessary to cover the set
		$\cA$. Then
		\begin{equation*}
		d(\cA; \B) = \limsup\limits_{\varepsilon \rightarrow 0} \frac{\log
			N_\B(\cA,\varepsilon)}{-\log \varepsilon}
		= \limsup_{\varepsilon\to 0}\; -\log_{{\varepsilon}} N_\B(\cA,\varepsilon)
		\end{equation*}
		denotes the upper box-counting dimension of $\cA$.
		\item[(c)]  Let $\B$ be a Banach space, and let $\cA\subset \B$ be compact.
		For $\varepsilon >0$,
		denote by $d_Y(\cA, \varepsilon)$ the minimal
		dimension of all finite dimensional subspaces $V\subset \B$ such
		that every point of $\cA$ lies within distance $\varepsilon$ of $V$;
		if no such $V$ exists, $d_Y(\cA, \varepsilon) = \infty$. Then
		\[
		\sigma(\cA, \B) = \limsup_{\varepsilon \to 0}
		\;-\log_{{\varepsilon}} d_Y(\cA, \varepsilon)
		\]
		is called the \emph{thickness exponent} of $\cA$ in $\B$.
	\end{enumerate}
	
\end{definition}

With these definitions the main results relevant for our framework are as follows:

\begin{theorem}[\cite{HuntKaloshin99}]
	Let $\B$ be a Banach space and $\cA\subset \B$ compact, with upper box counting dimension $d = d(\cA; \B)$ and thickness exponent $\sigma = \sigma(\cA, \B)$. Let $N>2d$ be an integer, and let $\alpha \in\R$ with
	\[ 0 < \alpha < \frac{ N-2d }{ N\cdot(1+\sigma) }. \]
	Then for a prevalent set of bounded linear maps $\cL:\B\to\R^N$ there is $C>0$ such that
	\[ C\cdot \| \cL(x-y)\|^\alpha \geq \|x-y\| \quad\text{for all }x,
	y\in \cA. \]
	\label{thm:HK99}
\end{theorem}
Note that this result implies that -- if $N$ is large enough --  a prevalent set of bounded linear maps ${\cL:\B\to\R^N}$ will be one-to-one on $\cA$.
This theorem lays the foundation for Robinson's main result concerning delay embedding techniques:

\begin{theorem}[\cite{R05}]\label{thm:R05}
	Let $Y$ be a Banach space and $\cA \subset Y$ a compact, invariant set,
	with upper box counting dimension $d$, and thickness exponent $\sigma$. Choose an integer $k> 2(1+\sigma)d$ and suppose further that the set $\cA_p$
	of $p$-periodic points of $\Phi$ satisfies $d(\cA_p;Y) < p/(2+2\sigma)$ for $p=1,\ldots,k$.
	Then for a prevalent set of Lipschitz maps $f:\B\rightarrow \R$ the observation map
	$D_k[f,\Phi] : \B \rightarrow \R^k$ defined by
	\begin{equation}\label{def:D_k}
	D_k[f,\Phi] (u) = \left(f(u),f(\Phi(u)),\ldots,f(\Phi^{k-1}(u))\right)^T
	\end{equation}
	is one-to-one on $\cA$.
\end{theorem}

\begin{remark}\label{rmk:ext_R05}\quad
	\begin{enumerate}
		\item[(a)] This result can be generalized to the case where several different observables are evaluated. More precisely, for a prevalent set of Lipschitz maps ${f_i:\B\to\R}$, $i=1,\dots,q\leq k$, the observation map ${D_k[f_1, \ldots, f_q, \Phi]:\B\to\R^k}$,
		\begin{equation*}
		u \mapsto \left(%
		f_1(u), \ldots, f_1\left(\Phi^{k_1-1}(u)\right),%
		\ldots, %
		f_q(u), \ldots, f_q\left(\Phi^{k_q-1}(u)\right)\right)^T
		\label{def:D_kvar}
		\end{equation*}
		is also one-to-one on $\cA$, provided that
		\[ 
		k =\sum_{i=1}^q k_i > 2(1+\sigma)\cdot d\  \mbox{ and }\  {d(\cA_p) < p/(2+2\sigma)} \quad \forall p\leq \max(k_1,\ldots, k_q).
		\]
		\item[(b)] As observed in \cite{HuntKaloshin99}, the thickness exponent is always bounded by the (upper) box-counting dimension, i.e.,
		\begin{equation*}
		\sigma(\cA, Y) \leq d(\cA; Y).
		\end{equation*}
		Thus, providing that we know the upper box-counting dimension $d$ of $\cA$ we can in principle always choose a \emph{worst-case} embedding dimension $k$ such that
		\begin{equation*}
		k > 2(1+d)d
		\end{equation*}
		is satisfied (cf. \cref{thm:R05}).
	\end{enumerate}
	
\end{remark}

\subsection{The core dynamical system (CDS)}
\label{sec:comp_emb_att}
In \cite{DHZ16} the results from \Cref{ssec:embedding_theory} have been used to create a finite dimensional dynamical system that allows to approximate invariant sets for infinite dimensional dynamical systems on Banach spaces. In this section we will briefly review the construction of the CDS.

Let $\cA$ be a compact invariant set of the infinite dimensional dynamical system \eqref{eq:DS}. We denote by $A_k$ the image of $\cA\subset \B$ under the \emph{observation map} $R:Y \rightarrow \R^k$, that is
\begin{align}\label{eq:A_k}
A_k = R(\cA),
\end{align}
where $R = \cL$ according to \cref{thm:HK99} or $R = D_k[f,\Phi]$ according to \cref{thm:R05}.
We then construct the CDS
\begin{equation}\label{eq:CDS}
x_{j+1} = \varphi(x_j), \quad j=0,1,2,\ldots,
\end{equation}
with $\varphi : \R^k \rightarrow \R^k$ as follows: At first we define $\varphi$ on the set $A_k$ by
\begin{equation*}
\label{def:phi}
\varphi = R \circ \Phi \circ \widetilde E,
\end{equation*}
where $\widetilde E:A_k \rightarrow \B$ is the continuous map satisfying
\begin{equation}
\label{eq:condRE}
(\widetilde E\circ R)(u) = u\quad \forall u\in \cA
\quad\text{ and }\quad
(R\circ \widetilde E)(x) = x\quad \forall x\in A_k.
\end{equation}
In fact, this is possible since $R$ is invertible as a mapping from $\cA$ to $A_k$. Observe that by this construction $A_k$ is an invariant set for $\varphi$. By using a generalization of Tietze's extension theorem by Dugundji \cite{Dugundji51} we extend $\widetilde{E}$ to a continuous map $E:\R^k\to \B$ with $E_{\vert A_k} = \widetilde E$ (see \cref{fig:kommdia}).
\begin{figure}[H]
	\centering
	\includegraphics[width=0.56\textwidth]{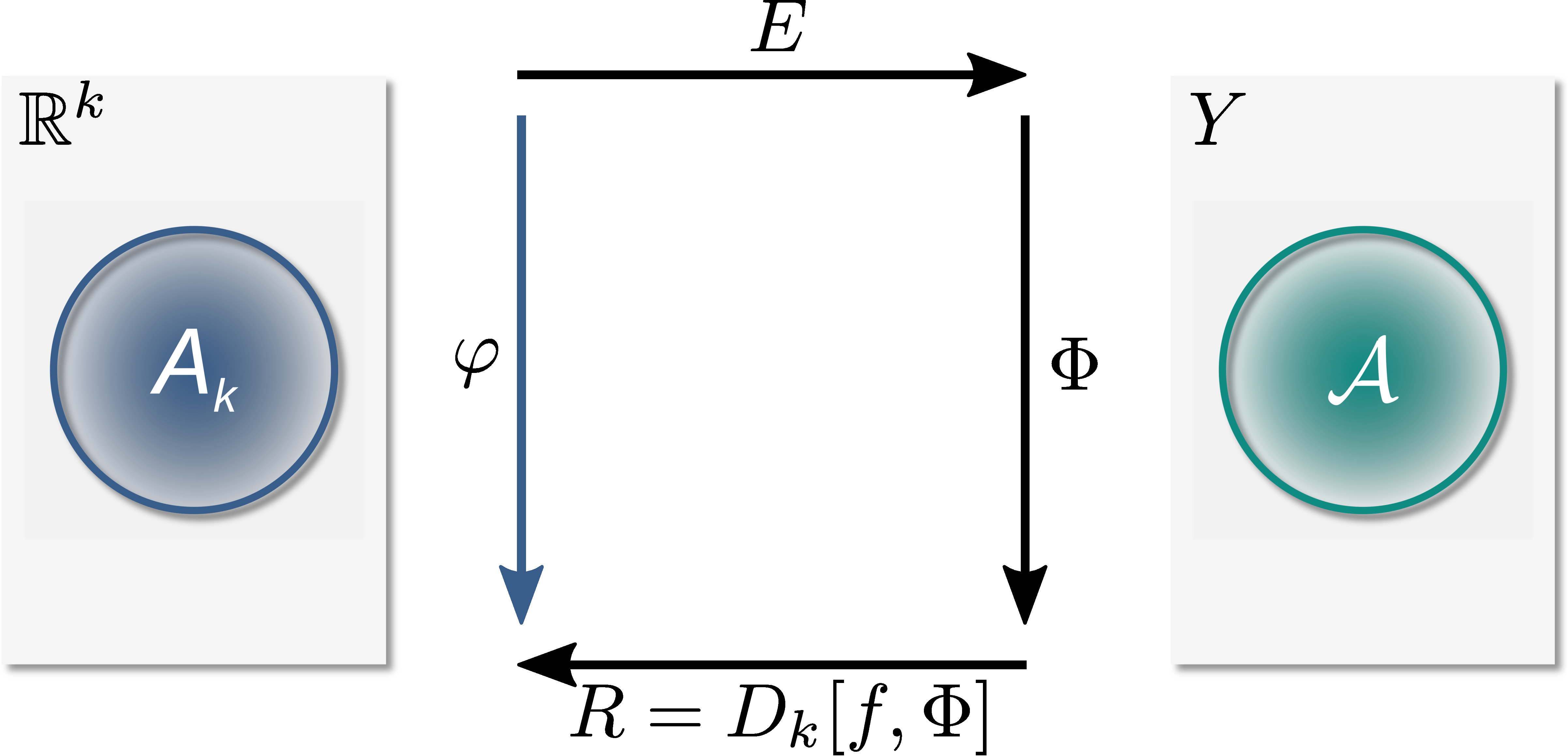}
	\caption{Definition of the CDS $\varphi$.}
	\label{fig:kommdia}
\end{figure}

Now we are in the position to extend the CDS $\varphi$ to $\R^k$:

\begin{proposition}
	\label{prop:phi_is_cont}
	There is a continuous map $\varphi:\R^k \to \R^k$ satisfying
	\begin{equation*}
	\varphi(R(u)) = R(\Phi(u)) \text{ for all } u\in \cA.
	\label{eq:conjugonA}
	\end{equation*}
\end{proposition}
For the proof the reader is referred to \cite{DHZ16}.

\subsection{Computation of embedded attractors via subdivision}
\label{sec:subdiv}
By our construction, the dynamics of the CDS $\varphi$ on $A_k$ is topologically conjugate to that of $\Phi$ on $\cA$. In what follows we will briefly review the main statements of \cite{DHZ16}.
We will use~\eqref{eq:CDS} in order to approximate the embedded invariant set $A_k$ (cf. \eqref{eq:A_k}) via subdivision. To this end, we give a brief review of the related subdivision scheme.

\vspace{1em}
\textbf{Subdivision scheme:}
\label{ssec:subdiv}
Let $Q \subset \R^k$ be a compact set. We define the \emph{global attractor relative to} $Q$ by
\begin{equation}
\label{eq:relativeAttractor}
A_Q = \bigcap_{j\ge 0} \varphi^j(Q).
\end{equation}
The aim is to approximate this set with the subdivision procedure.
By using the subdivision algorithm we obtain a sequence $\cB_0,\cB_1,\ldots$ of finite
collections of compact subsets of $\R^k$ such that the diameter
\[
\diam(\cB_\ell) = \max_{B\in\cB_\ell}\diam(B)
\]
converges to zero for $\ell\rightarrow\infty$.  Given an initial
collection $\cB_0$, we inductively obtain $\cB_\ell$ from $\cB_{\ell-1}$
for $\ell=1,2,\ldots$ in two steps.

\begin{enumerate}
	\item {\em Subdivision:} Construct a new collection
	$\hat\cB_\ell$ such that
	\begin{equation}
	\label{eq:sd1}
	\bigcup_{B\in\hat\cB_\ell}B = \bigcup_{B\in\cB_{\ell-1}}B
	\end{equation}
	and
	\begin{equation}
	\label{eq:sd2}
	\diam(\hat\cB_\ell) = \theta_\ell\diam(\cB_{\ell-1}),
	\end{equation}
	where $0<\theta_{\min} \le \theta_\ell\le \theta_{\max} < 1$.
	\item {\em Selection:} Define the new collection $\cB_\ell$ by
	\begin{equation}
	\label{eq:select}
	\cB_\ell=\left\{B\in\hat\cB_\ell : \exists \hat B\in\hat\cB_\ell
	~\mbox{such that}~\varphi^{-1}(B)\cap\hat B\ne\emptyset\right\}.
	\end{equation}
\end{enumerate}
The first step guarantees that the collections $\cB_\ell$ contains
successively finer sets for increasing $\ell$.  In fact, by construction
\begin{equation}\label{eq:diamB}
\diam(\cB_\ell)\leq\theta_{\max}^\ell\diam(\cB_0)\rightarrow 0\quad
\mbox{for $\ell\rightarrow\infty$.}
\end{equation}

In the second step we remove each subset whose preimage does neither
intersect itself nor any other subset in $\hat\cB_\ell$. 

Denote by $Q_\ell$ the collection of compact subsets obtained after $\ell$
subdivision steps, that is 
\[
Q_\ell=\bigcup_{B\in\cB_\ell}B.
\]
Observe that the $Q_\ell$'s define a nested sequence of compact sets, that is, $Q_{\ell+1}\subset Q_\ell$.
Therefore, for each $m$,
\begin{equation}
Q_m = \bigcap\limits_{\ell =1}^{m} Q_\ell,
\end{equation}
and we may view
\begin{equation}
Q_{\infty} = \bigcap\limits_{\ell =1}^{\infty} Q_\ell
\end{equation}
as the limit of the $Q_\ell$'s.
Then the selection step is responsible for the fact that the unions
$Q_\ell$ approach the relative global attractor:
\begin{proposition}\label{prop:AQ_Qinfty}
	Suppose that $A_Q$ satisfies $\varphi^{-1}(A_Q) \subset A_Q$. Then
	\[
	A_Q = Q_\infty.
	\]
\end{proposition}
Observe that in contrast to \cite{DH97} we have to assume that $\varphi^{-1}(A_Q) \subset A_Q$ since the CDS is only continuous and not homeomorphic. Moreover, we note that we can, in general, not expect that $A_k = A_Q$. In fact, by construction $A_Q$ may contain several invariant sets and related heteroclinic connections. However, if $\cA$ is an attracting set we can prove equality:
\begin{proposition}\label{prop:convergence_embedded_invariant_sets}\quad
	\begin{enumerate}
		\item[(a)]  Let $A_Q$ be the global attractor relative to the compact set $Q$,
		and suppose that the embedded attractor $A_k$ satisfies $A_k \subset Q$.
		Then
		\begin{equation}
		A_k \subset A_Q.
		\end{equation}
		\item[(b)] Suppose that $\cA$ is an attracting set with basin of attraction $\cU \supset \cA$ and choose $Q \subset \R^k$ such that $A_k \subset Q$ and $E(Q) \subset \cU$. Define for $m \geq 1$ the continuous maps
		\begin{equation*}
		\varphi_m = R\circ \Phi^m \circ E
		\end{equation*}
		and denote the corresponding relative global attractors by $A_Q^m$, where
		\begin{equation*}
		A_Q^m = \bigcap_{j \geq 0} \varphi_m^j(Q).
		\end{equation*}
		Then
		\begin{equation*}
		A_k = A_Q^\infty,
		\end{equation*}
		where $A_Q^\infty = \bigcap_{m\geq 1} A_Q^m$. 
	\end{enumerate}
\end{proposition}

\begin{remark}
	Roughly speaking \cref{prop:convergence_embedded_invariant_sets}~(b) states that it is possible to approximate an attracting set for $\Phi$ if we perform the computations with appropriately high iterates of $\Phi$.
\end{remark}

\section{A subdivision and continuation technique for embedded unstable manifolds}
\label{sec:continuation}
In this section we extend the results of \cite{DH96} for the computation of finite dimensional unstable manifolds of infinite dimensional dynamical systems of the form \eqref{eq:DS}. We will in particular focus on
invariant manifolds for semiflows on Banach spaces (cf. \cite{Ha71,He06,Ca12,CH12}).

\subsection{Approximation of the local unstable manifold}

Let us denote by
\begin{equation}\label{eq:unstable_manifold}
\cW_{\Phi}^u(u^*) \subset \cA
\end{equation}
the unstable manifold of $u^* \in \cA$, where $u^*$ is a steady state solution of the infinite dimensional dynamical system $\Phi$ (cf. \eqref{eq:DS}). Furthermore, let us define the \emph{embedded unstable manifold} $W^u(p)$ by
\begin{equation}\label{eq:embedded_unstable_manifold}
W^u(p) = R(\cW_{\Phi}^u(u^*)) \subset A_k,
\end{equation}
where $p = R(u^*)$ and $R$ is the observation map introduced in \Cref{sec:review}. Observe that, by construction, $W^u(p)$ is an invariant set for $\varphi$ (cf. \eqref{eq:condRE}). Our goal is to approximate compact subsets of $W^u(p)$ or even the entire closure $\overline{W^u(p)}$ via an adaptation of a set-oriented continuation method introduced in \cite{DH96}. 

Let us denote by $\cW_{\Phi,loc}^u(u^*) \subset \cA$ the local unstable manifold of the steady state $u^*$ and choose a compact neighborhood ${C\subset A_k}$ such that
\begin{equation*}
W_{loc}^u(p) = R(\overline{\cW_{\Phi,loc}^u(u^*)}) \subset C.
\end{equation*}
Since $R$ is a homeomorphism on $\cA$ we can conclude that $W_{loc}^u(p)$ is compact. Then the following proposition holds: 
\begin{proposition}\label{prop:elle}\quad
	\begin{enumerate}
		\item[(a)] Let $A_C$ be the global attractor relative to $C$. Then
		\begin{equation*}
		W_{loc}^u(p) \subset A_C.
		\end{equation*}
		\item[(b)] Let us suppose that $\overline{\cW_{\Phi,loc}^u(u^*)}$ is a compact attracting set with basin of attraction ${\cU \supset \overline{\cW_{\Phi,loc}^u(u^*)}}$. Choose $C \subset \R^k$ such that $W_{loc}^u(p) \subset C \subset A_k$ and $E(C) \subset \cU$. Then
		\begin{equation*}
		W_{loc}^u(p) = A_C.
		\end{equation*}
	\end{enumerate}
\end{proposition}
\begin{proof}
	\begin{enumerate}
		\item[(a)] By assumption $W_{loc}^u(p) \subset A_k$ and by construction of the CDS (cf. \eqref{eq:condRE}) it is easy to see that $W_{loc}^u(p) \subset \varphi\left(W_{loc}^u(p)\right)$. Thus, by \cite[Lemma 4.1]{DHZ16} it follows that
		\begin{equation*}
		W_{loc}^u(p) \subset A_C.
		\end{equation*}
		\item[(b)] Define for $m \geq 1$ the continuous maps
		\begin{equation*}
		\varphi_m = R\circ \Phi^m \circ E
		\end{equation*}
		and denote the corresponding relative global attractors by $A_C^m$, where
		\begin{equation*}
		A_C^m = \bigcap_{j \geq 0} \varphi_m^j(C).
		\end{equation*}
		By \cref{prop:convergence_embedded_invariant_sets}~(b) we obtain
		\begin{equation*}
		W_{loc}^u(p) = A_C^\infty.
		\end{equation*}
		It remains to show that $A_C^\infty = A_C$. Since $C \subset A_k$, it is easy to see that
		\begin{equation*}
		\varphi_m(C) = \varphi^m(C) \quad \mbox{for all } m \in \N.
		\end{equation*}
		(cf. \eqref{eq:condRE}). Thus,
		\begin{equation*}
		A_C^\infty = \bigcap_{m\geq 1} \bigcap_{j \geq 0} \varphi_m^j(C) = \bigcap_{m\geq 1} \bigcap_{j \geq 0} \varphi^{jm}(C) = \bigcap_{i\geq 0}\varphi^{i}(C) = A_C.
		\end{equation*} 
	\end{enumerate}
\end{proof}

\begin{remark}\quad
	\begin{enumerate}
		\item[(a)] Observe that \cref{prop:elle}~(b) states that the embedding of the local unstable manifold (in infinite dimensional space) is identical to the local embedded unstable manifold.
		\item[(b)] If the steady state $u^*\in\cA$ is hyperbolic, then  $\overline{\cW_{\Phi,loc}^u(u^*)}$ is attractive since by assumption its dimension is finite (cf. \eqref{eq:unstable_manifold}). In particular, the compact set $C$ can be chosen, such that the assumed properties are satisfied.
	\end{enumerate}
\end{remark}

From now on we assume that the assumptions of \cref{prop:elle}~(b) are satisfied. The idea of the continuation algorithm is to globalize the local covering of $W_{loc}^u(p)$ in order to obtain an approximation of the entire embedded unstable manifold $W^u(p)$. 

\subsection{The continuation method}
The continuation starts at $p = R(u^*)$ of the embedded unstable manifold $W^u(p)$. We choose a compact set $Q \subset \R^k$ containing $p$ and we assume that $Q$ is large enough so that it contains the entire embedded unstable manifold of $p$, i.e.,
\begin{equation}
\label{eq:closure}
\overline{W^u(p)} \subset Q.
\end{equation}
We remark that this assumption can be relaxed, and we will discuss this point later in the context of the realization of the approximation scheme.

In order to combine the subdivision process with a continuation method, we realize the subdivision using a family of partitions of $Q$. We define a partition $\cP$ of $Q$ to be a finite family of compact subsets of $Q$ such that
\[
\bigcup_{B \in \cP}B = Q\quad \mbox{and}\quad \mbox{int} B\cap \mbox{int} B' = \emptyset,\ \mbox{for all } B, B'\in \cP,~ B\neq B'.
\]
Moreover, we denote by $\cP(x) \in \cP$ the element of $\cP$ containing $x\in Q$. We consider a nested sequence $\cP_s,\ s \in \N$, of successively finer partitions of $Q$, requiring that for all $B \in \cP_s$ there exist $B_1,\ldots,B_m \in \cP_{s+1}$ such that $B = \cup_i B_i$ and ${\diam(B_i)\leq\theta\diam(B)}$ for some $0 < \theta < 1$. A set $B \in \cP_s$ is said to be of \emph{level} $s$. 

In what follows, we assume that ${C = \cP_s(p) \subset A_k}$ for $s$ sufficiently large such that $p\in \mbox{int }C$. The purpose is to approximate subsets $W_j \subset W^u(p)$ where $W_0 = W_{loc}^u(p)$ and
\[
W_{j+1} = \varphi(W_j) \quad \mbox{for $j = 0,1,2,\ldots$}
\]
Here $\varphi$ is the CDS (see \eqref{eq:CDS}).

The numerical realization of the continuation algorithm for the approximation of embedded unstable manifolds can be described as follows:

\begin{algorithm}[!htb]
	\caption{The continuation method for embedded unstable manifolds}\label{alg:continuation}
	\vspace{1em}
	{\em Initialization:} Given $k>2(1+\sigma)d$ we choose an initial box $Q \subset \R^k$, defined by a $k$-dimensional generalized rectangle of
	the form
	\[
	Q(c,r) = \left\{ y \in \R^k: |y_i - c_i|\leq r_i \mbox{ for } i=1,\ldots,k \right\},
	\]
	where $c,r \in \R^k,\ r_i > 0$ for $i=1,\ldots,k$, are the center and the radii, respectively. Choose a partition $\cP_s$ of $Q$ and a box $C \in \cP_s$ such that $p = R(u^*) \in C$.
	\begin{enumerate}
		\item Apply the subdivision algorithm with $\ell$ subdivision steps to $\cB_0 = \{ C \}$ to obtain a covering $\cB_\ell \subset \cP_{s+\ell}$ of the local embedded unstable manifold $A_{C}$. 
		\item Set \[ C_0^{(\ell)} = \cB_\ell. \]
		\item For $j=0,1,2,\ldots$ define
		\begin{align}\label{eq:continuation}
		C_{j+1}^{(\ell)} = \left\lbrace B \in \cP_{s+\ell}: \exists B' \in C_j^{(\ell)} \mbox{ such that }B\cap \varphi(B') \neq \emptyset\right\rbrace.
		\end{align}
	\end{enumerate}
\end{algorithm}

\newpage
\begin{remark}\label{rem:nested}\quad
	\begin{enumerate}
		\item[(a)]
		Observe that the unions
		\[
		C_j^{(\ell)} = \bigcup_{B \in \cC_j^{(\ell)}} B
		\]
		form a nested sequence in $\ell$, i.e.,
		\[
		C_j^{(0)}\supset C_j^{(1)} \supset \ldots \supset C_j^{(\ell)} \ldots .
		\]
		In fact, it is also a nested sequence in $j$, i.e.,
		\[
		C_0^{(\ell)}\subset C_1^{(\ell)} \ldots \subset C_j^{(\ell)} \ldots.
		\]
		\item[(b)] Due to the compactness of $Q$ the continuation in Step 3 of \cref{alg:continuation} will terminate after finitely many, say $J_\ell$, steps. We denote the corresponding box covering obtained by the continuation
		method by
		\begin{equation}
		\label{eq:zeta}
		G_\ell = \bigcup_{j=0}^{J_\ell} C_j^{(\ell)} = C_{J_\ell}^{(\ell)}.
		\end{equation}
	\end{enumerate}
\end{remark}

In \cref{fig:continuation} we illustrate the continuation method described in \cref{alg:continuation}.

\begin{figure}[h]
	\begin{minipage}{0.325\textwidth}
		\includegraphics[width = \textwidth]{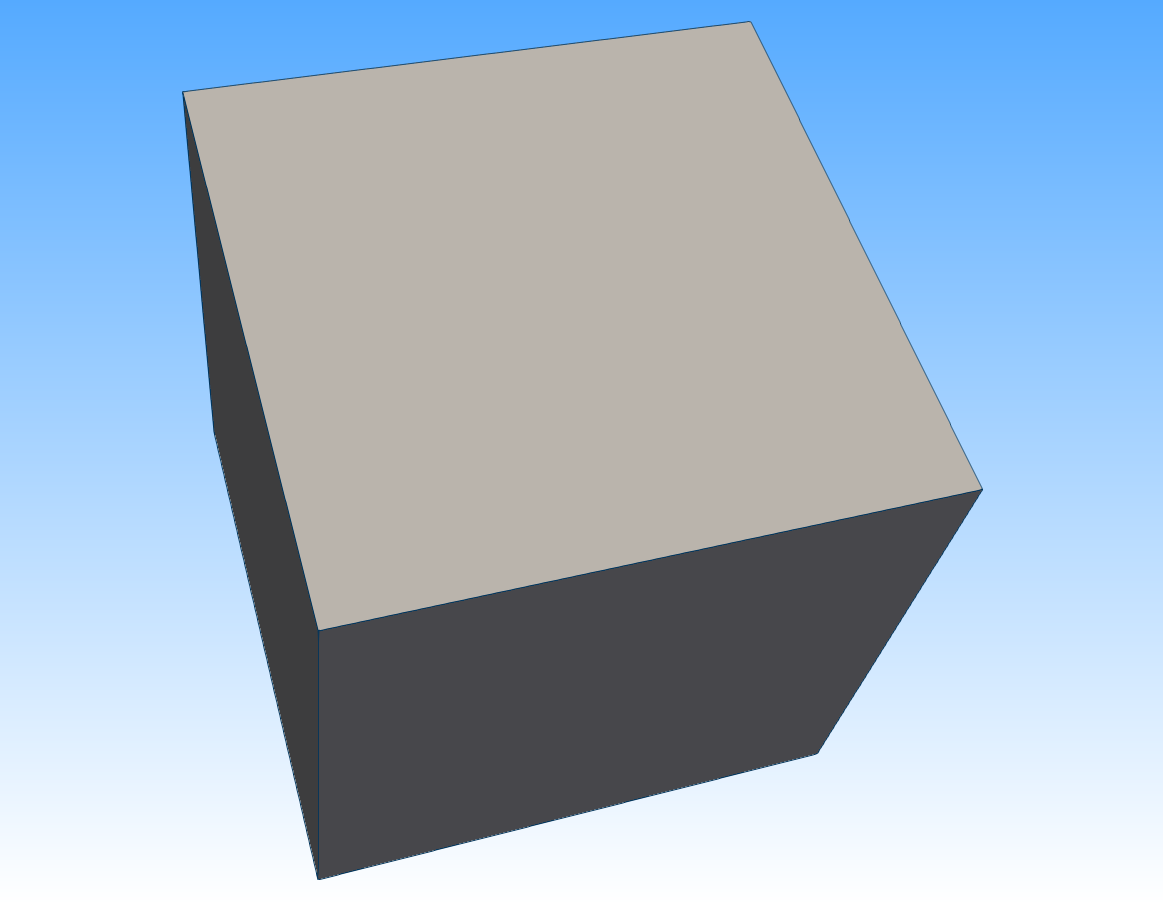}
		\centering \scriptsize{(a)}
	\end{minipage}
	\hfill
	\begin{minipage}{0.325\textwidth}
		\includegraphics[width = \textwidth]{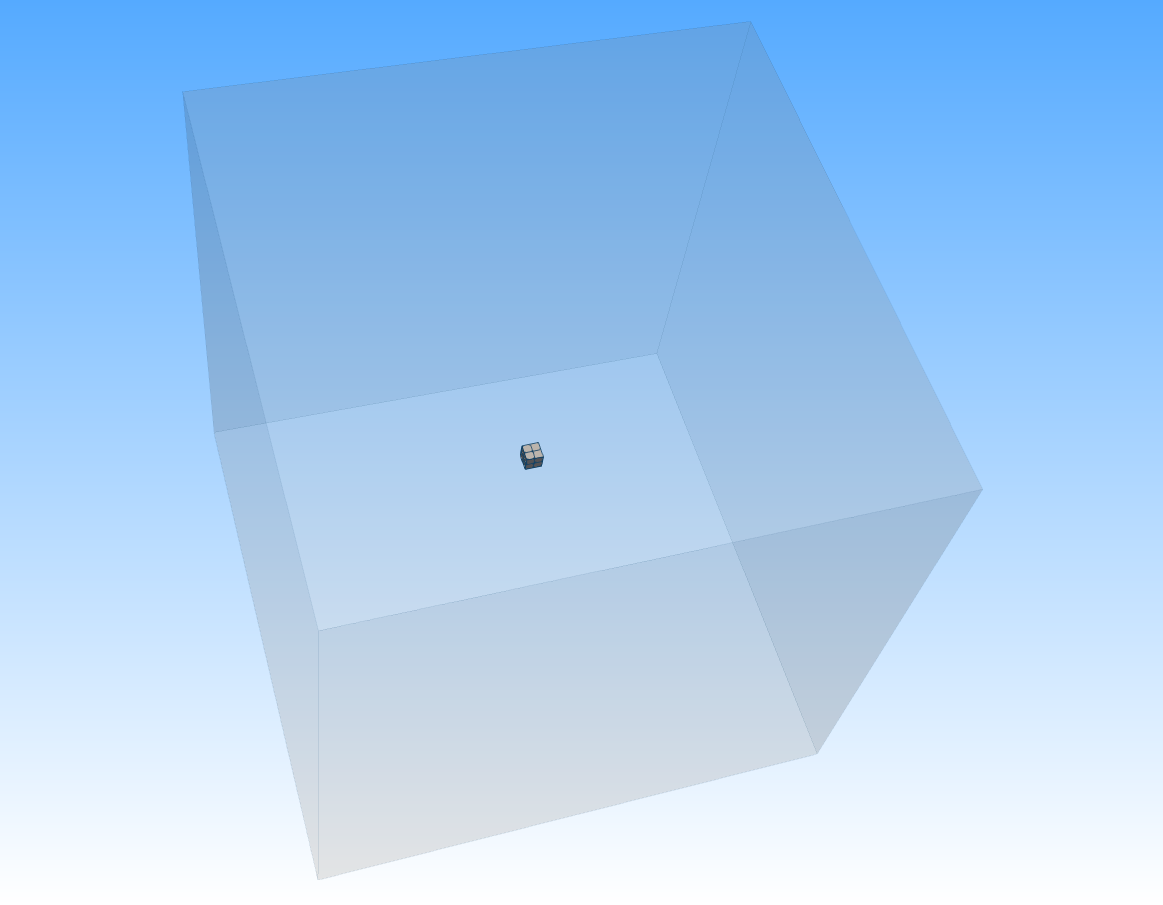}
		\centering \scriptsize{(b)}
	\end{minipage}
	\hfill
	\begin{minipage}{0.325\textwidth}
		\includegraphics[width = \textwidth]{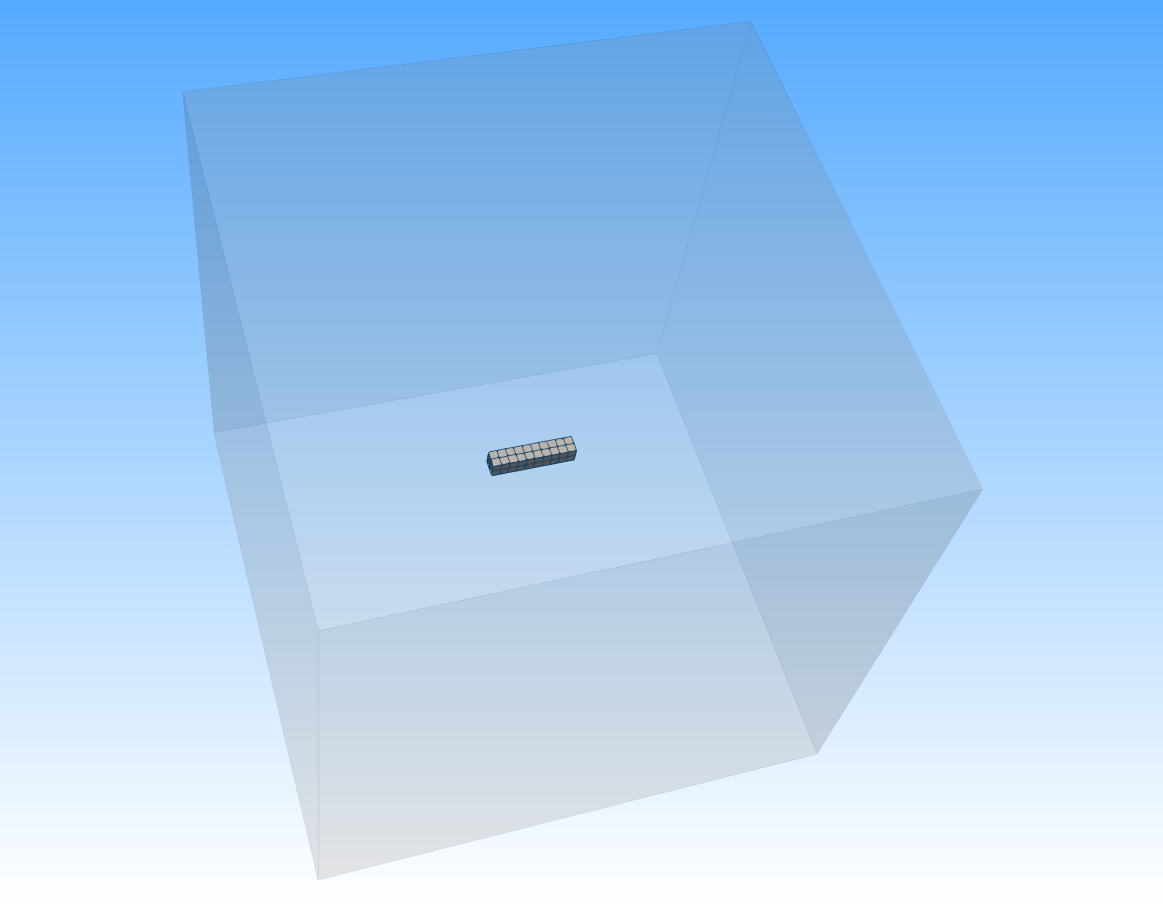}
		\centering \scriptsize{(c)}
	\end{minipage}\\[1em]
	\begin{minipage}{0.325\textwidth}
		\includegraphics[width = \textwidth]{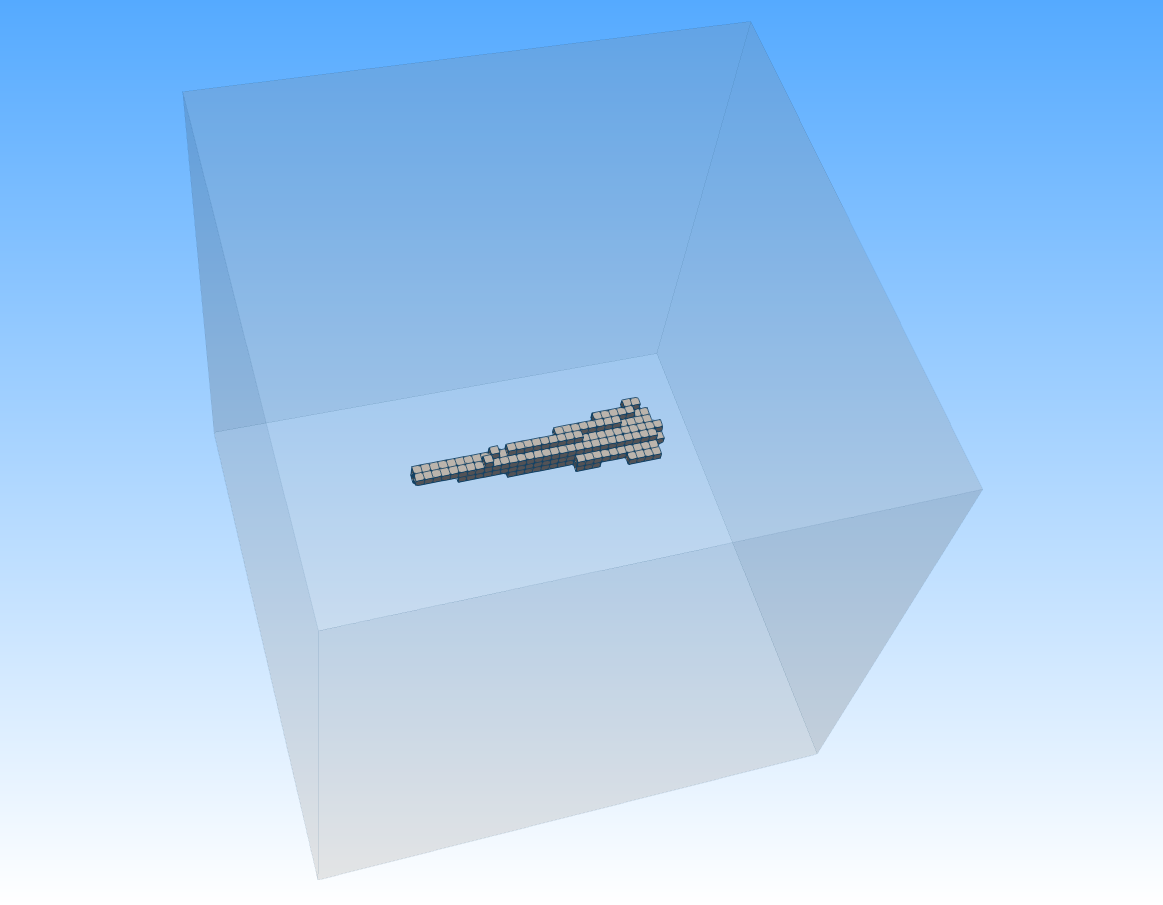}
		\centering \scriptsize{(d)}
	\end{minipage}
	\hfill
	\begin{minipage}{0.325\textwidth}
		\includegraphics[width = \textwidth]{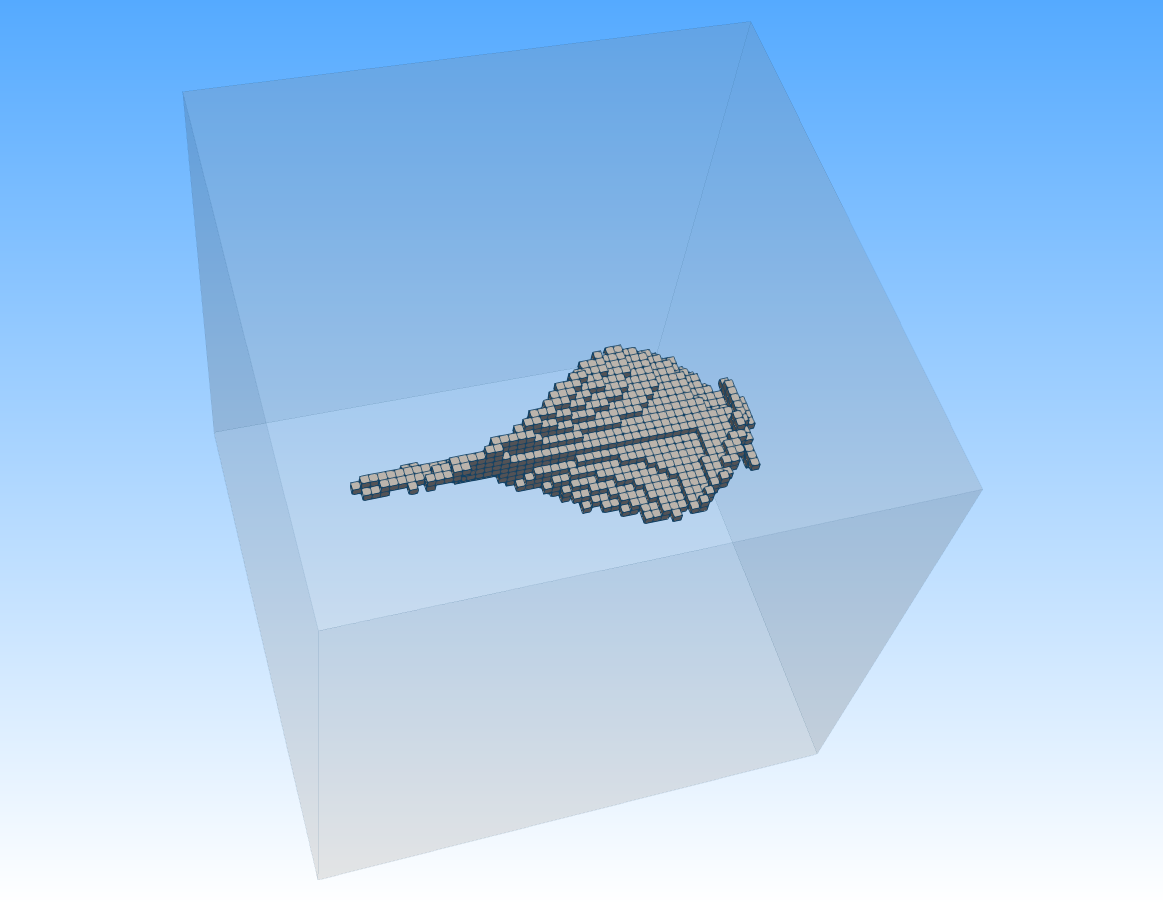}
		\centering \scriptsize{(e)}
	\end{minipage}
	\hfill
	\begin{minipage}{0.325\textwidth}
		\includegraphics[width = \textwidth]{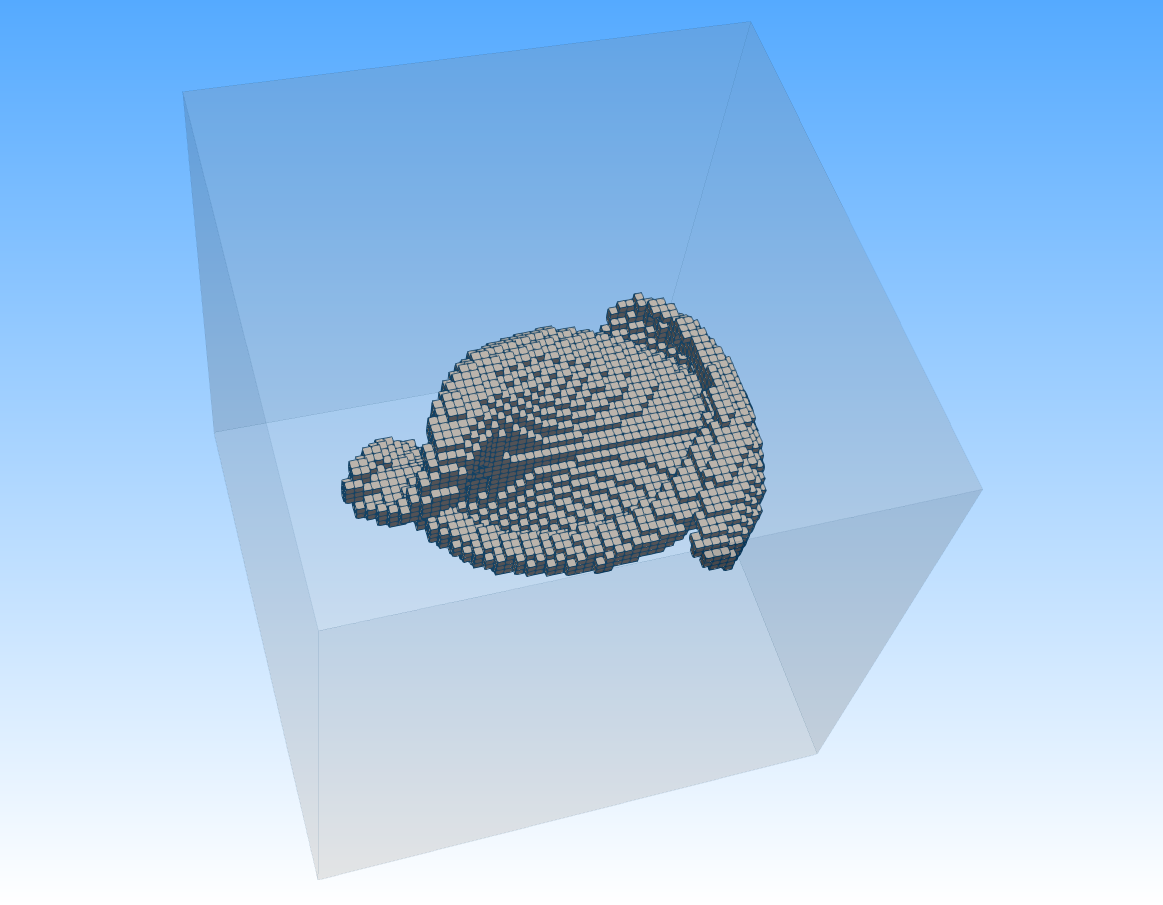}
		\centering \scriptsize{(f)}
	\end{minipage}
	\caption{(a) Initial box $Q \supset A_k$. (b) Let $C \in \cP_{s}$ be the box containing $p=R(u^*)$ ($\ell = 0$). (c) First continuation step. (d)-(f) Repeat step (c) with those boxes that are contained in the previous covering until no additional boxes were marked. A detailed discussion of the underlying dynamical system can be found in \Cref{sssec:oseberg}.}
	\label{fig:continuation}
\end{figure}

Intuitively it is clear that the algorithm, as constructed, generates an approximation of the embedded unstable manifold $W^u(p)$. In particular, we expect that the bigger $s$ and $\ell$ are the better the approximation will be.

\begin{proposition}\label{prop:convergenceCC}\quad
	\begin{enumerate}
		\item[(a)] The sets $C_j^{(\ell)}$ are coverings of $W_j$ for all $j,\ell = 0,1,\ldots$. Moreover, for fixed $j$, we have
		\[
		\bigcap_{\ell = 0}^\infty C_j^{(\ell)} = W_j.
		\]
		\item[(b)] Suppose that $\overline{W^u(p)}$ is linearly attractive, i.e., there is a $\lambda \in (0,1)$ and a neighborhood $U\supset Q \supset \overline{W^u(p)}$ such that
		\begin{equation}\label{eq:basin}
		\setdist{\varphi(y),\overline{W^u(p)}}\leq \lambda~\setdist{y,\overline{W^u(p)}} \quad\forall y\in U.
		\end{equation}
		Then the box coverings obtained by \cref{alg:continuation} converge to the closure of the embedded unstable manifold $\overline{W^u(p)}$. That is,
		\[
		\bigcap_{\ell = 0}^\infty  G_\ell=\overline{W^u(p)}.
		\]
	\end{enumerate}
\end{proposition}
\begin{proof}
	\begin{enumerate}
		\item[(a)] Applying the subdivision algorithm with $\ell$ subdivision steps to the initial covering ${\cB_0 = \{C\}}$, we obtain $\cB_\ell \subset \cP_{s+\ell}$ of $A_C$, that is,
		\[
		A_C\subset \bigcup_{B \in \cB_\ell} B.
		\]
		Here we assume that the subdivision scheme is creating coverings using elements from the partitions $\cP_n$. 
		By \cref{prop:elle}~(b) $A_C = W_{loc}^u(p)$ and by \cref{prop:AQ_Qinfty} the box coverings converge to $A_C$ for $\ell\rightarrow \infty$. Therefore, $A_C = W_{loc}^u(p) = W_0$ for $\ell \to \infty$. Since $j$ is fixed a continuity argument shows that the sets $C_j^{(\ell)}$ converge to $W_j$ for $\ell \rightarrow \infty$, i.e.,
		\[
		\bigcap_{\ell = 0}^\infty C_j^{(\ell)} = W_j.
		\]
		\item[(b)] For each $\ell$ \cref{alg:continuation} yields a covering of $\overline{W^u(p)}$, and therefore
		\[
		\bigcap_{\ell = 0}^\infty  G_\ell\supset\overline{W^u(p)}.
		\]
		Suppose there is $x\in \bigcap_{\ell = 0}^\infty  G_\ell \setminus \overline{W^u(p)}$. Since $\overline{W^u(p)}$ is compact, it follows that $\setdist{x,\overline{W^u(p)}}>0$.
		By definition of $x$, \cref{alg:continuation} generates a $\mbox{diam}(\cB_{\ell})$-pseudo orbit $\{x_0,\ldots,x_{j(\ell)}\}$ for each $\ell\geq 0$, where $x_{j(\ell)} = x$. That is
		\[
		x_{j}\in C_{j}^{(\ell)} \mbox{ and } \varphi(x_{j})\in \cP_{s+\ell}(x_{j+1}) \quad\forall j\in\{0,\ldots,j(\ell)-1\}.
		\]
		Here $\cP_{s+\ell}(x_{j+1})\subset C_{j+1}^{(\ell)}$ denotes the element of $\cP_{s+\ell}$ containing $x_{j+1}\in C_{j+1}^{(\ell)}$ and $j(\ell)=\min\{j\in\{0,\ldots,J_\ell\}:x\in C_j^{(\ell)}\}$, i.e., for each $j\geq j(\ell)$ continuation steps $x$ is covered. Observe that the sequence $j(\ell)$ is monotonically increasing in $\ell$ (cf. Step 3 of \cref{alg:continuation} and \cref{rem:nested}) and
		\begin{align}\label{eq:box_diam}
		\|{x_j-\varphi(x_{j-1})}\|\leq \diam(B_{\ell}) \quad \forall j\in \{0,\ldots,j(\ell)-1\}.
		\end{align}
		We first show by contradiction that $j(\ell)$ is unbounded. Suppose that $j(\ell)$ is bounded by some $J\in \N_0$, i.e., $\max{j(\ell)}=J$. Hence, by monotony of $j(\ell)$ there is $\ell_0\in \N_0$ such that $j(\ell)=J$ for all $\ell\geq \ell_0$. Using \cref{rem:nested}~(a) we have
		\[
		x\in \left(\bigcap_{\ell=0}^{\ell_0-1} C_{j(\ell)}^{(\ell)}\right) \cap \left(\bigcap_{\ell = \ell_0}^\infty C_J^{(\ell)}\right) \subset \bigcap_{\ell = \ell_0}^\infty C_J^{(\ell)}=\bigcap_{\ell = 0}^\infty C_J^{(\ell)}.
		\]
		However, by \cref{prop:convergenceCC}~(a) it follows that $x\in W_J\subset \overline{W^u(p)}$ which is a contradiction to $\setdist{x,\overline{W^u(p)}}>0$. Thus, $j(\ell)$ is unbounded.

		By assumption $\overline{W^u(p)}$ is linearly attractive in a neighborhood $U$. Hence, we can use \eqref{eq:basin} and \eqref{eq:box_diam} on the $\mbox{diam}(\cB_{\ell})$-pseudo orbit $\{x_0,\ldots,x_{j(\ell)}\}$, where $x_{j(\ell)} = x$, in combination with the triangle inequality to obtain
		\begin{align*}
		\setdist{x,\overline{W^u(p)}}&\leq \setdist{\varphi(x_{j(\ell)-1}),\overline{W^u(p)}}+\diam(\cB_\ell)\\
		&\leq \lambda~\setdist{x_{j(\ell)-1},\overline{W^u(p)}}+\diam(\cB_\ell)\\
		&\ \vdots\\
		&\leq \lambda^{j(\ell)}~\setdist{x_0,\overline{W^u(p)}}+\diam(\cB_\ell)\sum_{i=0}^{j(\ell)-1}\lambda^i\\
		&\leq \lambda^{j(\ell)}~\setdist{x_0,\overline{W^u(p)}}+\frac{\diam(\cB_\ell)}{1-\lambda} \quad\longrightarrow 0 \mbox{ for } \ell \rightarrow\infty.
		\end{align*}
		Here the last expression converges to zero because
		$\lambda\in (0,1)$
		and $\diam(\cB_\ell)$ converges to zero for $\ell\rightarrow \infty$ (see \eqref{eq:diamB}). Again we have an contradiction to $\setdist{x,\overline{W^u(p)}}>0$. It follows that
		\[
		\bigcap_{\ell = 0}^\infty  G_\ell\subset\overline{W^u(p)},
		\]
		which yields the desired statement.
		
	\end{enumerate}
\end{proof}

\begin{remark}\label{rmk:assum}\quad
	\begin{enumerate}
		\item[(a)] If \eqref{eq:closure} would not be satisfied, then it
		can in general not be guaranteed that the continuation method leads to an approximation
		of the entire set $\overline{W^u(p)}$ or even ${\overline{W^u(p)}\cap Q}$. Rather it has to be expected that this is not the case. The reason is that the embedded unstable manifold $W^u(p)$ may 'leave' $Q$ but may as well 'wind back' into it. In this scenario the continuation method, as described above, will not cover all of $\overline{W^u(p)}\cap Q$. This situation is illustrated in \cref{fig:hyperbolicFixpoint}~(a).
		\item[(b)] The assumption in \cref{prop:convergenceCC}~(b) is, for instance, not satisfied
		if $\overline{W^u(p)}$ forms a heteroclinic connection between
		the steady state solution $p$ and another unstable hyperbolic steady state $q$.
		In fact, in this case the algorithm would also generate a covering of the embedded unstable manifold of $q$. This situation is illustrated in \cref{fig:hyperbolicFixpoint}~(b).
		\item[(c)] If \eqref{eq:basin} is not satisfied, but $\overline{\cW_{\Phi}^u(u^*)}$ is attractive, one can apply the subdivision scheme introduced in \cite{DHZ16} to $G_\ell$ in order to approximate $\overline{W^u(p)}$ more accurately (cf. \cref{prop:convergence_embedded_invariant_sets}~(b)).
		\item[(d)] Note that \eqref{eq:basin} is satisfied if the observation map $R$ is bi-Lipschitz with Lipschitz constant $L\geq 1$ and $\overline{\cW_{\Phi}^u(u^*)}$ is linearly attractive with $\lambda \in (0,L^{-2})$.
	\end{enumerate}
\end{remark}

\begin{figure}[htb]
	\begin{minipage}{0.425\textwidth}
		\includegraphics[width = \textwidth]{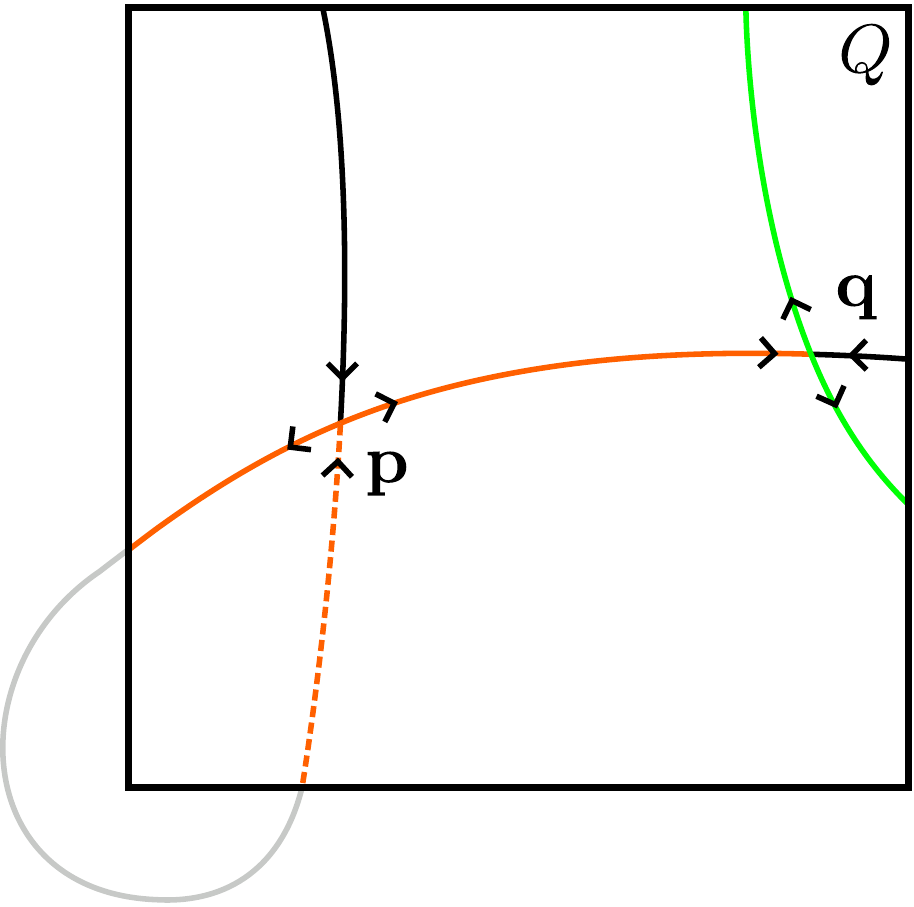}
		\centering (a)
	\end{minipage}
	\hfill
	\begin{minipage}{0.425\textwidth}
		\includegraphics[width = \textwidth]{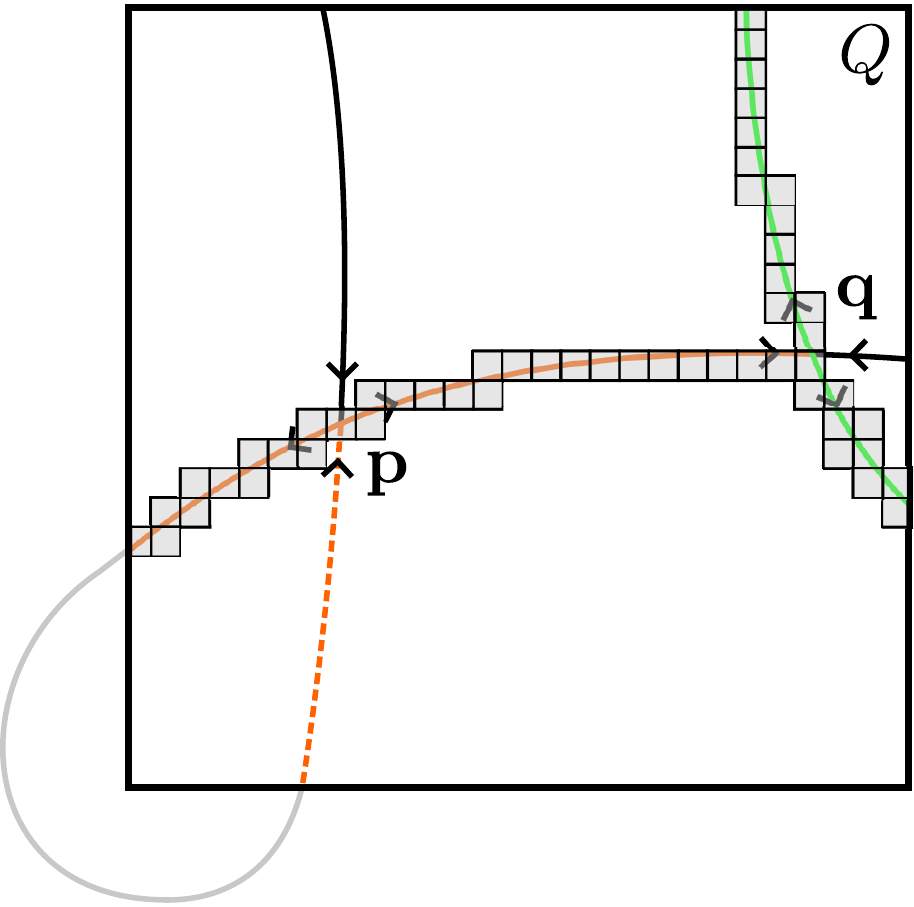}
		\centering (b)
	\end{minipage}
	\caption{Illustration of the possible situations discussed in \cref{rmk:assum}~(a) and (b). \newline
		(a) The dashed line will not be covered by the continuation method and thus, we will not approximate the entire set $W^u(p)\cap Q$. (b)	Schematic box covering obtained by the continuation method, where in this particular case we also obtain a covering of $W^u(q)\cap Q$.}
	\label{fig:hyperbolicFixpoint}
\end{figure}

\section{Numerical realization of the CDS $\varphi$ for partial differential equations}
\label{sec:num_real_PDE}

As discussed in the introduction, dynamical systems with infinite dimensional state space, but finite dimensional attractors arise in particular in two areas of applied mathematics, namely dissipative partial differential equations and delay differential equations with small constant time delay. In this article we will focus on the PDE case, and we will present one specific realization of the maps $R$ and $E$ for this situation.

More precisely, we will consider explicit differential equations of the form
\begin{align}
\frac{\partial}{\partial t} u(y,t)=F(y,u),~u(y,0)=u_0(y),
\label{eq:PDE}
\end{align}
where $u:\R^n\times \R \rightarrow \R^n$ is in some Banach space $Y$ and $F$ is a (nonlinear) differential operator. We assume that the dynamical system $\eqref{eq:PDE}$ has a well-defined semiflow on $Y$.

In order to numerically realize the construction of the map $\varphi = R\circ\Phi\circ E$ described in \cref{sec:comp_emb_att}, we have to work on three tasks: the implementation of $E$, the implementation of $R$, and the realization of the time-$T$-map of \eqref{eq:PDE}, denoted by $\Phi$, respectively. For the latter we will rely on standard methods for forward time integration of PDEs, e.g., a fourth-order time stepping method for the one-dimensional Kuramoto-Sivashinsky equation \cite{KT05}. Observe that the numerical realization of $\Phi$ strongly depends on the underlying PDE. The map $R$ will be realized on the basis of \cref{thm:HK99} or \cref{rmk:ext_R05}, respectively. For the numerical construction of the continuous map $E$ we will present a new method that uses statistical information from previous computations. This step is in particular crucial for the continuation step, since we want to restart the algorithm with initial conditions that satisfy the identities
\begin{equation*}
(E\circ R)(u) = u\quad \forall u\in  \mathcal{W}_\Phi^u(u^*)
\quad\text{ and }\quad
(R\circ E)(x) = x\quad \forall x\in W^u(p).
\end{equation*}
(cf. \eqref{eq:condRE}) at least approximately.

From now on we assume that upper bounds for both the box counting dimension $d$ and the thickness exponent $\sigma$ are available.
This allows us to fix $k > 2(1+\sigma)d$ according to \cref{rmk:ext_R05}.
\subsection{Numerical realization of $R$}
\label{ssec:num_real_R}

In \cite{DHZ16} $R$ has been defined on the basis of \cref{thm:R05}, where function evaluations at a fixed time have been used. Thus, in the case of scalar delay differential equations $R$ has been defined as the delay coordinate map
\begin{align}\label{eq:R_delay}
R = D_k[f,\Phi] (u) = (u(-\tau),\Phi(u)(-\tau),\ldots,\Phi^{k-1}(u)(-\tau))^T,
\end{align}
where $\tau>0$ is the constant time-delay of the underlying DDE. In principle it would also be possible to observe the evolution of a partial differential equation by a delay coordinate map. However, from a computational point of view this would be very inefficient. The reason is that for the realization of the map ${E:\R^k \to Y}$ one would have to reconstruct functions from time delay coordinates. Thus, for each point in observation space one would essentially have to store the entire corresponding function. To overcome this problem, in this work we will present a different approach. In what follows, we will assume that the function $u \in Y$ can be represented in terms of an orthonormal basis $\{\Psi_i\}_{i=1}^\infty$, i.e.,
\begin{align}\label{eq:u_basis}
u(y,t) = \sum_{i=1}^{\infty} x_i(t) \Psi_i(y),
\end{align}
where the $\Psi_i$ are elements from a Hilbert space (e.g., $L^2$).
Then our observation map $R$ will be defined by projecting a function onto $k$ coefficients $x_i$ of its Galerkin expanion. For the approximation of $u$, i.e.,
\begin{align}\label{eq:u_galerkin}
u(y,t) \approx \sum_{i=1}^{S} x_i(t) \Psi_i(y)
\end{align}
we want to use an optimal basis $\{ \Psi_i \}_{i=1}^S$ (i.e., as small as possible) in the sense that it contains the \emph{'most characteristic'} data from an ensemble of functions. The notion 'most characteristic' implies the use of an averaging operation. Furthermore, this basis has to be capable of representing the solution $u \in Y$ of the underlying PDE \eqref{eq:PDE} with a small error. Both requirements can be addressed by the \emph{proper orthogonal decomposition} (POD) (c.f.~\cite{Sir87,BHL93,HLB+12}), also known as the \emph{principal component analysis} or the \emph{Karhunen-Lo{\`e}ve transformation}. In order to compute a basis $\{\Psi_i\}$ for $i=1,\ldots,S$ and $S > k$, we first generate time-snapshots of a long-time simulation for some set of initial conditions of the underlying PDE. Then we approximate the POD-basis via the singular value decomposition (e.g., \cite{Chat00, LLL+02,Volk11}). Using this basis we then approximate $u \in Y$ by \eqref{eq:u_galerkin} where $x_i(t)$ denotes the $i$-th POD-coefficient at time $t$.

Given the basis $\{ \Psi_i \}_{i=1}^S$ and using the fact that this basis is orthogonal, we then define the observation map by choosing $k$ different observables
\begin{equation}
f_i(u) = \langle u, \Psi_i \rangle = x_i \quad \mbox{for }i=1,\ldots,k.
\end{equation}
This yields
\begin{equation}\label{eq:numR}
R(u) = (f_1(u), \ldots, f_k(u))^\top = (x_1, \ldots, x_k)^\top.
\end{equation}
Observe that $R$ is linear and bounded and hence, for $k$ sufficiently large, \cref{thm:HK99} and \cref{rmk:ext_R05}, respectively, guarantee that generically (in the sense of prevalence) $R$ will be a one-to-one map on $\cA$.

\subsection{Numerical realization of $E$}
\label{ssec:num_real_E}
In the application of the continuation scheme for the computation of embedded unstable manifolds $W^u(p)$ described in \cref{sec:continuation} one has to perform the continuation step
\begin{align*}
\cC_{j+1}^{(\ell)} = \left\lbrace B \in \cP_{s+\ell}: \exists B' \in \cC_j^{(\ell)} \mbox{ such that }B\cap \varphi(B') \neq \emptyset\right\rbrace
\end{align*}
(see \eqref{eq:continuation}). Numerically this is realized as follows: At first $\varphi$ is evaluated for a large number of test points $x \in B'$ for each box $B' \in \cC_j^{(\ell)}$. Then a box $B \in \cP_{s+\ell}$ is added to the collection
$\cC_{j+1}^{(\ell)}$ if there is a least one $x \in B'$ such that $\varphi(x)\in B$.

\begin{remark}\label{rmk:testpoints}
	In practice the test points $x\in B'$ can be chosen according to several different
	strategies: In low dimensional problems one can choose them
	from a regular grid within each box $B$. Alternatively one can select the test points from the boundaries of the boxes. In our computations we use a Monte Carlo sampling.
\end{remark}

By \cref{ssec:num_real_R} the state space for the CDS $\varphi$ is given by points $x \in \R^k$ where $x_1,\ldots,x_k$ are the POD-coefficients. For the evaluation of $\varphi = R\circ \Phi \circ E$ at a test point $x$ we need to define the image $E(x)$, that is, we need to generate adequate initial conditions for the forward integration of the PDE \eqref{eq:PDE}. In the first step of the continuation method we proceed as follows. Given the POD-basis $\{\Psi_i\}$ for $i = 1,\ldots,S$ and $S>k$, we simply construct initial conditions $u = E(x)$ near the unstable (hyperbolic) steady state by defining the map $E$ as
\begin{equation}\label{eq:numE}
E(x) = \sum_{i=1}^{k} x_i \Psi_i.
\end{equation}
Observe that by this choice of $E$ and $R$ both conditions $(R\circ E)(x) = x$ and ${(E \circ R)(u) = u}$ are satisfied for each test point $x \in \cC_0^{(\ell)}$ (see \eqref{eq:condRE} and \eqref{eq:numR}). 
Then we use the time-$T$-map $\Phi$ of the underlying PDE to obtain a function $\bar u = \Phi(E(x))$. If $k$ is not sufficiently large, then $(E\circ R)(\bar u) = \bar u$ will in general not be satisfied anymore. Therefore, it is possible that initial conditions $u=E(x)$ for $x \in \cC_{j+1}^{(\ell)}$, $j=0,1, \ldots$, generated by \eqref{eq:numE} are not even close to the unstable manifold. This is not acceptable since the requirement $(E\circ R)(u) = u$ for all $u\in \cW_\Phi^u(u^*)$ (see \eqref{eq:condRE}) is crucial in order to compute a reliable covering of the embedded unstable manifold. 

To enforce this equality at least approximately we extend the expansion and construct initial functions by
\begin{equation}\label{eq:E_extended}
E(x) = \sum_{i=1}^{k} x_i\Psi_i + \sum_{l=k+1}^{S}x_l\Psi_l.	
\end{equation}
Here only the first $k$ POD-coefficients are given by the coordinates of points inside $B \subset \R^k$. Thus, it remains to discuss how to choose the POD-coefficients $x_{k+1},\ldots,x_S$. The idea is to use a new heuristic strategy that utilizes statistical information obtained in the previous continuation step:
Suppose we want to evaluate $\varphi$ for a large number of test points $x$ in a box $B\in \cC_{j+1}^{(\ell)}$. By the continuation step (cf.~\cref{alg:continuation}), there must have been at least one ${\hat B\in\cC_{j}^{(\ell)}}$ such that $\bar x = R(\Phi(E(\hat x)))\in B$ for at least one test point $\hat x\in \hat B$. For all these points $\bar x$ we can compute the POD-coefficients $\bar x_{k+1},\ldots, \bar x_S$ by
\begin{align*}
\bar x_i = \langle \Phi(E(\hat x)), \Psi_i \rangle, \quad i=k+1,\ldots,S.
\end{align*}
Then we sample the box $B$ with all points $\bar x$ for which additional information is available. However, the number of these points $\bar x$ might be too small, such that $B$ is not discretized sufficiently well and we have to generate additional test points. For this, we first choose a certain number of points $\widetilde x \in B$ at random. Then we extend these points to elements in $\R^S$ as follows: We first compute componentwise the mean value $\mu_i$ and the variance $\sigma_i^2$ of all POD-coefficients $\bar x_{i}$, for $i = k+1,\ldots,S$. This allows us to make a Monte Carlo sampling for the additional coefficients of $\widetilde x_i$ for $i=k+1,\ldots,S$, i.e.,
\begin{equation*}
\widetilde{x}_i \sim \cN(\mu_i,\sigma_i^2) \quad \mbox{ for }i=k+1,\ldots,S.
\end{equation*}
Finally, we compute initial functions of the form
\[
E(\widetilde x) = \sum_{i=1}^{S} \widetilde x_i\Psi_i.
\] 
By this construction we expect in each continuation step to generate initial functions that satisfy an approximation of the identity $(E\circ R)(u) = u$ for all $u\in \cW_\Phi^u(u^*)$. We will illustrate our statistical approach in the next section for the one-dimensional Kuramoto-Sivashinsky equation.

\section{Numerical results}
\label{sec:numex}
In this section we present results of computations carried out for the Mackey-Glass delay differential equation and the Kuramoto-Sivashinsky equation, respectively. For the numerical realization of the CDS $\varphi$ for delay differential equations the reader is referred to \cite{DHZ16}.

\subsection{The Kuramoto-Sivashinsky equation}
We start with the well-known  \\
\mbox{Kuramoto-Sivashinsky} equation in one spatial dimension which is given by
\begin{equation}
\begin{aligned}\label{eq:KS}
&u_t + \nu u_{yyyy} + u_{yy} + \frac{1}{2}(u_y)^2 = 0, \quad 0 \leq y \leq L,\\
&u(y,0) = u_0(y), \quad u(y+L,t) = u(y,t).
\end{aligned}
\end{equation}
This equation has been studied extensively over the past $40$ years. It has, for instance, been used to model phase dynamics in reaction-diffusion systems \cite{KT76} or small thermal diffusive instabilities in laminar flame fronts \cite{Siv77}. Following \cite{HNZ86,KNS90}, we normalize the K-S equation to an interval length of $2\pi$ and set the damping parameter to the original value derived by Sivashinsky, i.e., $\nu = 4$. Then equation \eqref{eq:KS} can be written as

\begin{equation}
\begin{aligned}\label{eq:KS_normalized}
&u_t + 4u_{yyyy}+\mu\left[ u_{yy} + \cfrac 1 2 (u_y)^2 \right] = 0, \quad 0\leq y\leq 2\pi,\\
&u(y,0) = u_0(y), \quad u(y+2\pi,t) = u(y,t).
\end{aligned}
\end{equation}
In equation \eqref{eq:KS_normalized} we introduce a new parameter $\mu = L^2/4\pi^2$, where $L$ denotes the size of a typical pattern scale (cf. \eqref{eq:KS}). In \cite{HNZ86,KNS90} numerical and analytical studies were made by varying $\mu$ over a finite interval, showing the  complex hierarchy of bifurcations.
We are in particular interested in computing the unstable manifold of the trivial unstable steady state for different parameter values $\mu$. In order to use our algorithm developed in \cref{sec:continuation} it is crucial to have a good estimate of the dimension of the invariant set $\cA$ and $\cW_\Phi^u(y)$, respectively (cf. \cref{sec:review}). In \cite{RO94} it has been shown that the dimension of the inertial manifold of \eqref{eq:KS} for $\nu = 1$ is $d \leq L^{2.46}$, i.e., each invariant set has finite dimension. However, these estimates are very pessimistic and we expect that we will obtain one-to-one images of the unstable manifold for smaller related embedding dimensions $k$. 

In what follows, the observation space is defined through projections onto the first $k$ POD-coefficients. For each parameter value $\mu$ we compute the POD-basis by using the snapshot-matrix obtained through a long-time integration with the initial condition
\[
u_0(y) = 10^{-4}\cdot \cos\left(y\right)\cdot \left(1+\sin\left(y\right)\right)
\]
(cf. \cref{ssec:num_real_R}).
\begin{figure}[t!]
	\parbox[b]{0.49\textwidth}{\centering \includegraphics[width=0.49\textwidth]{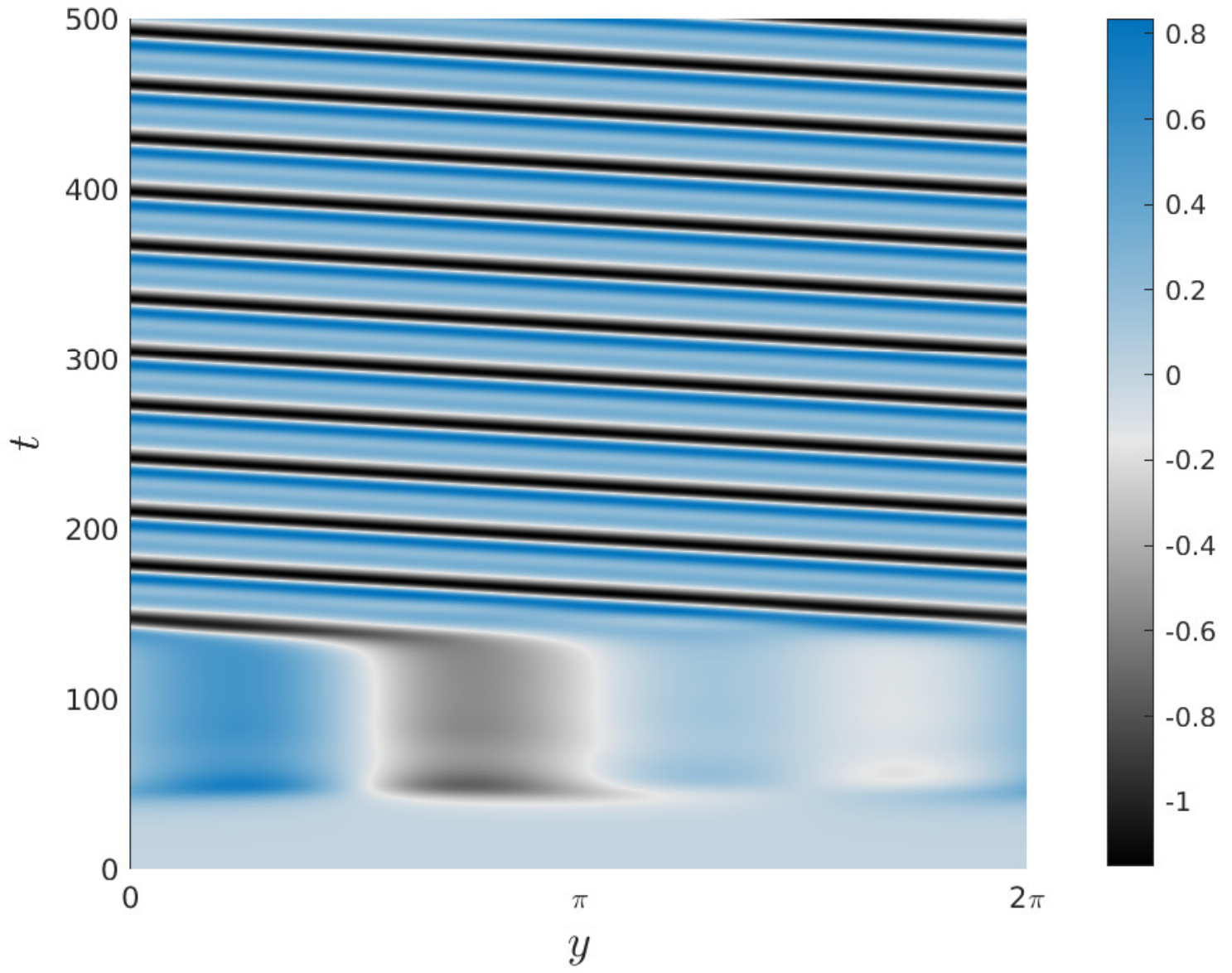}\\\scriptsize (a) Direct simulation}
	\parbox[b]{0.49\textwidth}{\centering \includegraphics[width=0.49\textwidth]{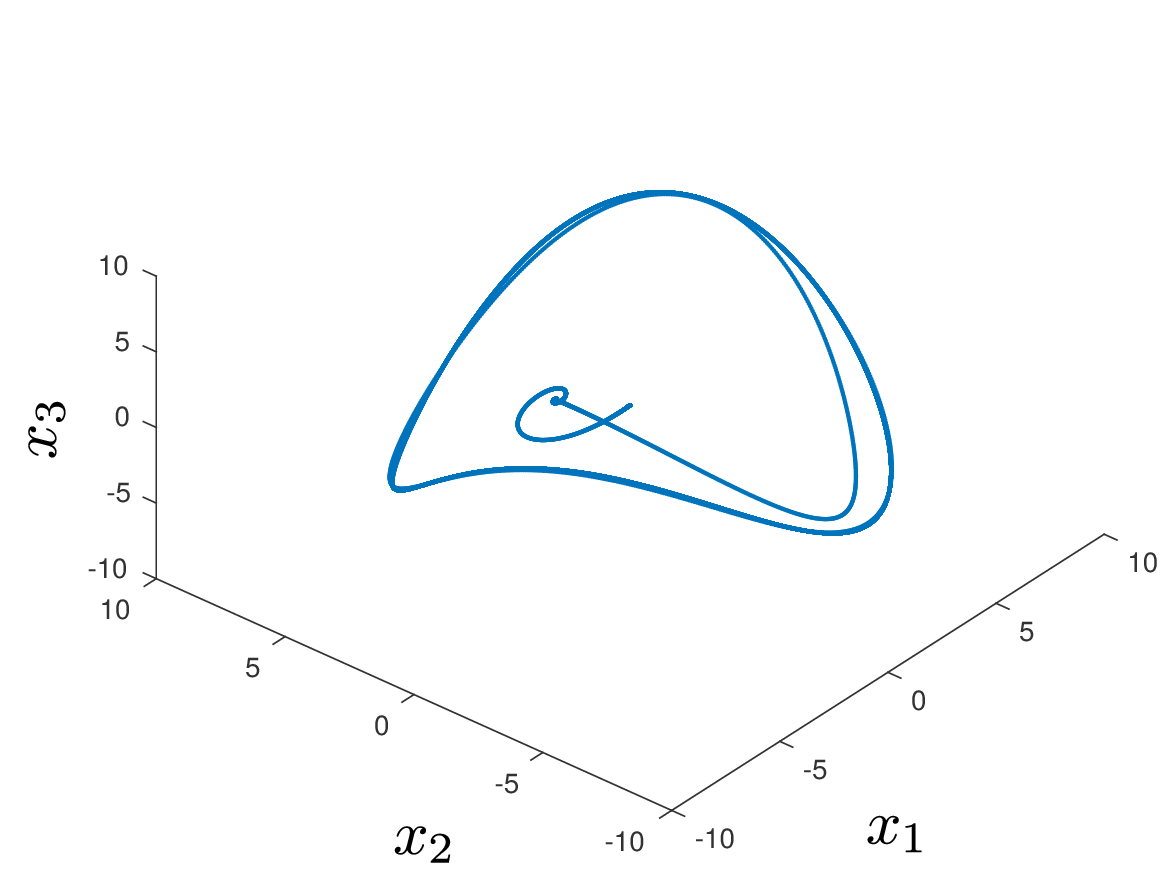}\\\scriptsize (b) Observation space}
	\caption{(a) Direct simulation of the Kuramoto-Sivashinsky equation for $\mu = 15$. The initial value is attracted to a traveling wave solution;\ (b) Corresponding embedding in observation space. As expected the CDS possesses a related limit cycle.}
	\label{fig:KS15_simulation}
\end{figure}

\subsubsection{The traveling wave}\quad\\
For the parameter value $\mu = 15$ the Kuramoto-Sivashinsky equation possesses a stable traveling wave solution (cf.~\cref{fig:KS15_simulation}~(a)). Due to the symmetry imposed by the periodic boundary conditions there are two waves traveling in opposite directions \cite{KNS90}. Correspondingly the CDS possesses for each traveling wave a limit cycle in observation space (cf. \cref{fig:KS15_simulation}~(b)). We expect that the dimension of the embedded unstable manifold is approximately two since different initial conditions result in trajectories in observation space that are rotations of each other about the origin.

By choosing the embedding dimension $k = 7$ we restrict the initial functions in the first continuation step to the subspace that is spanned by the first seven POD-modes and since $d(\overline{W^u(p)};\R^k) \approx 2$, we expect to obtain a one-to-one image of $\cW_{\Phi}^u(u^*)$. In \cref{fig:newE} we show a comparison of the numerical realization of the continuous map $E$ with initial functions generated by \eqref{eq:numE}, i.e., where $x_{k+1} = \ldots = x_S = 0$ for all test points $x$ (red) and the statistical approach discussed in \cref{ssec:num_real_E} (blue). By using only $k=3$ POD-coefficients, we construct initial functions that by far do not satisfy $(E\circ R)(u) = u$. 
\begin{figure}[!htb]
	\centering
	\includegraphics[width=.6\textwidth]{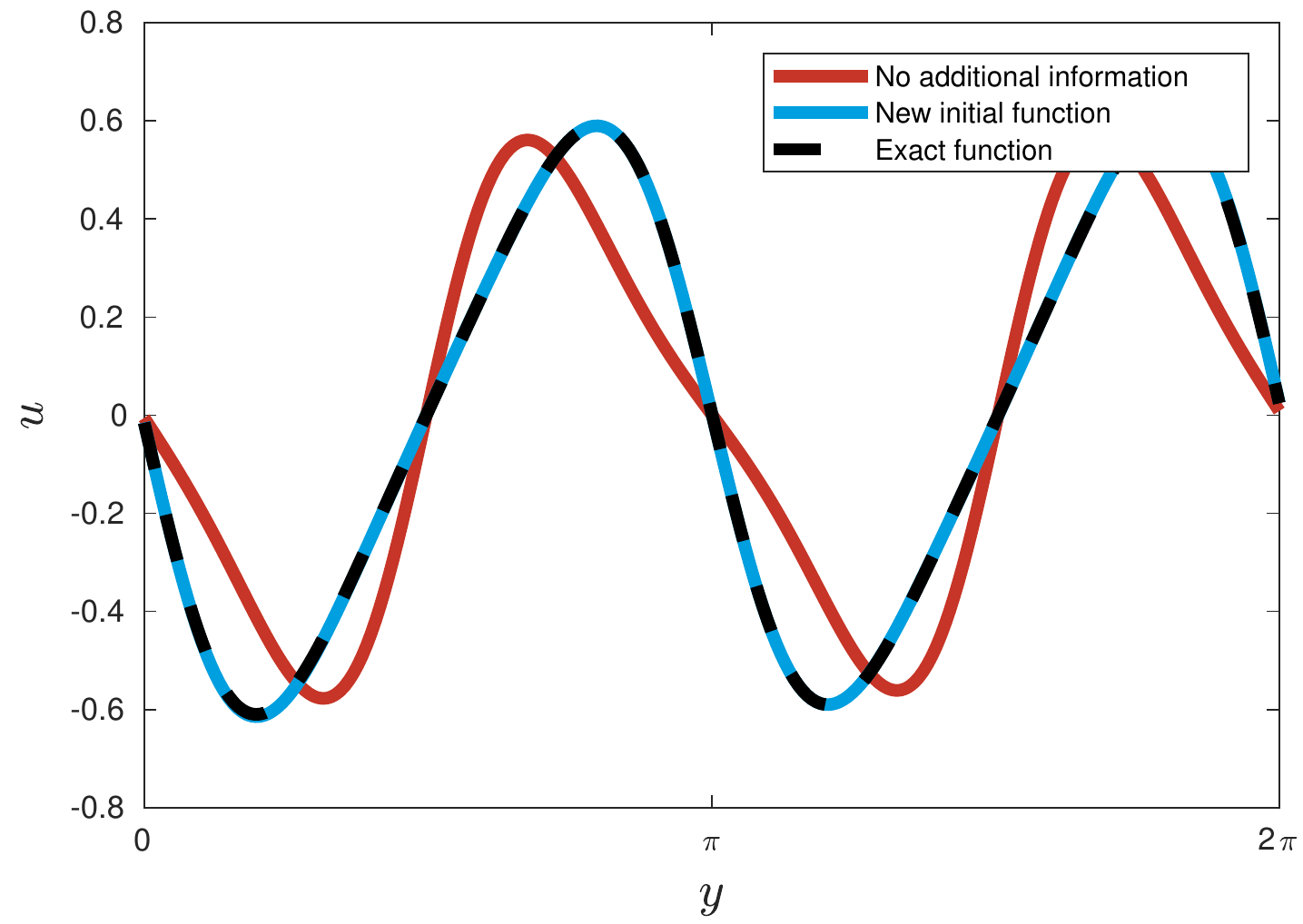}
	\caption{Illustration of the numerical realization of $E$ for the Kuramoto-Sivashinsky equation for $\mu = 15,\ k=3$ and $S = 13$ in one specific box $B$ after at least one continuation step: The black function represents $\Phi(E(\hat x))$ from the previous continuation step; the red function is generated by \eqref{eq:numE}; the blue function is generated by \eqref{eq:E_extended}, where we use additional statistical information for the POD-coefficients $x_{k+1},\ldots, x_S$.}
	\label{fig:newE}
\end{figure}

We choose $Q = [-8,8]^7$ and initialize a fine partition $\cP_s$ of $Q$ for $s = 21, 35, 49, 63$. 
Next we set $T = 200$. In addition, we define a finite time grid $\{ t_0,\ldots, t_N \}$, where $t_N = T$, and mark all boxes that are hit in each time step (a similar approach has been used in \cite{Junge99}). This strategy will be used for each example in this section. 

In \cref{fig:KS_15}~(a)-(d) we illustrate successively finer box coverings of the unstable manifold as well as a transparent box covering depicting the complex internal structure of the unstable manifold. Observe that the boundary of the unstable manifold consists of two limit cycles which are symmetric in the first POD-coefficient $x_1$. This is due to the fact that the Kuramoto-Sivashinsky equation with periodic boundary conditions  \eqref{eq:KS_normalized} possesses $O(2)$-symmetry.

\begin{figure}[!h]
	\begin{minipage}{0.48\textwidth}
		\includegraphics[width = \textwidth]{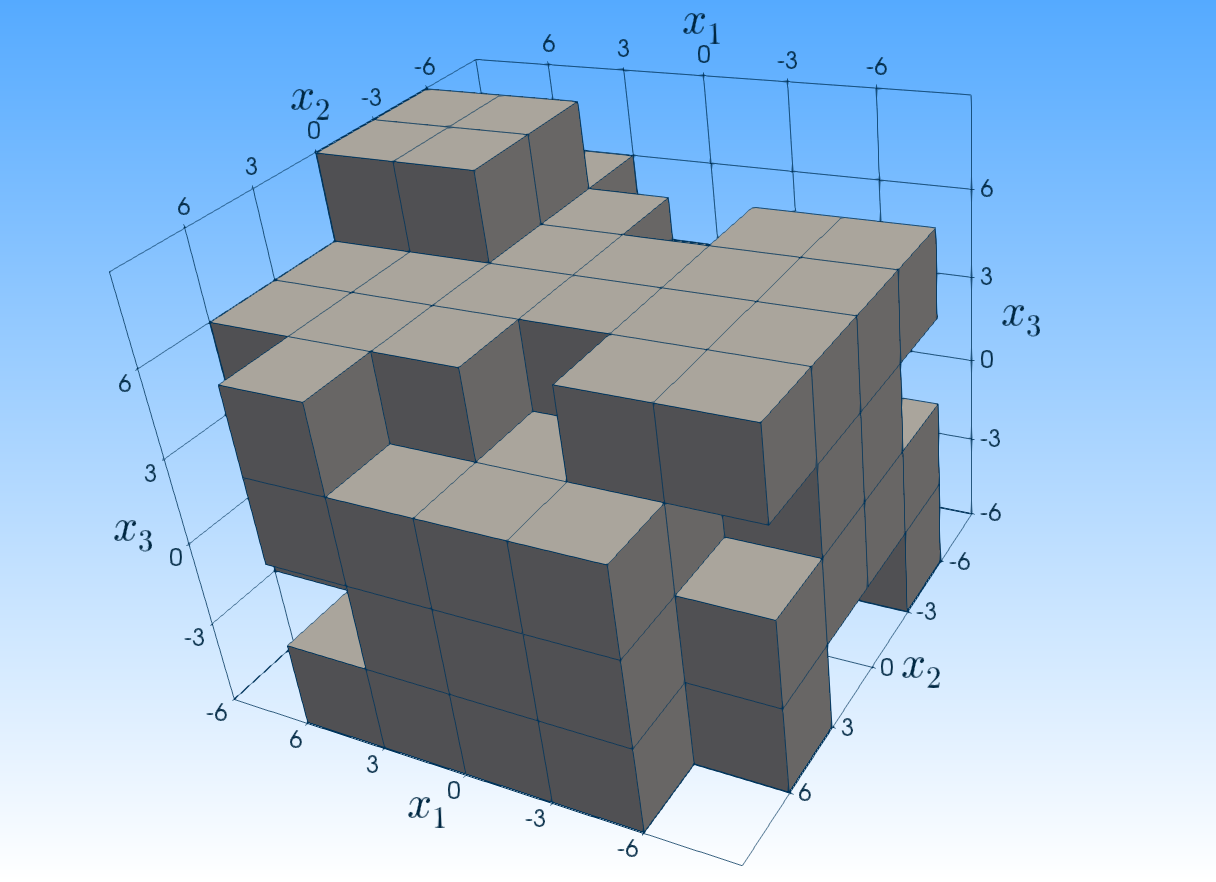}\\
		\centering \scriptsize{(a) $s = 21$ and $\ell = 0$}
	\end{minipage}
	\hfill
	\begin{minipage}{0.48\textwidth}
		\includegraphics[width = \textwidth]{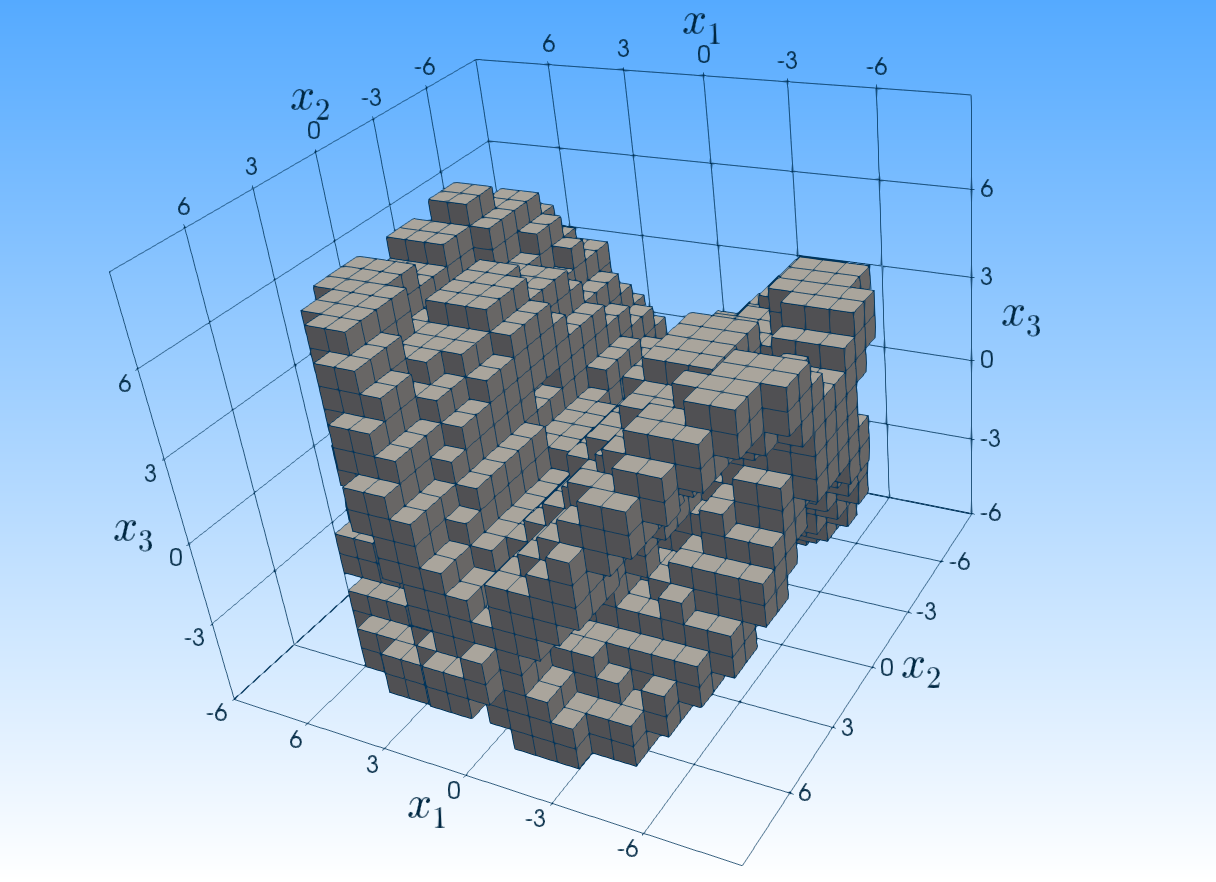}\\
		\centering \scriptsize{(b) $s = 35$ and $\ell = 0$}
	\end{minipage}\\[1em]
	\begin{minipage}{0.48\textwidth}
		\includegraphics[width = \textwidth]{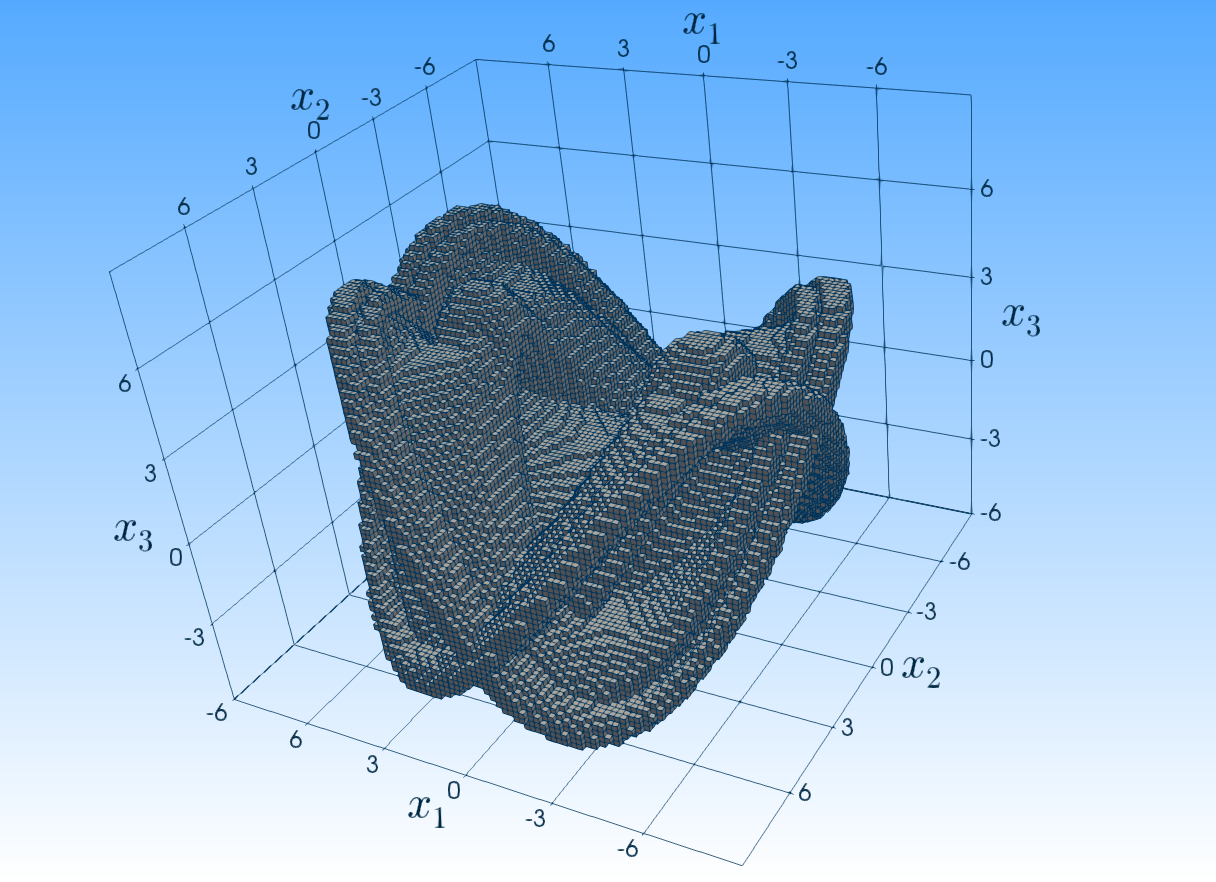}\\
		\centering \scriptsize{(c) $s = 49$ and $\ell = 0$}
	\end{minipage}
	\hfill
	\begin{minipage}{0.48\textwidth}
		\includegraphics[width = \textwidth]{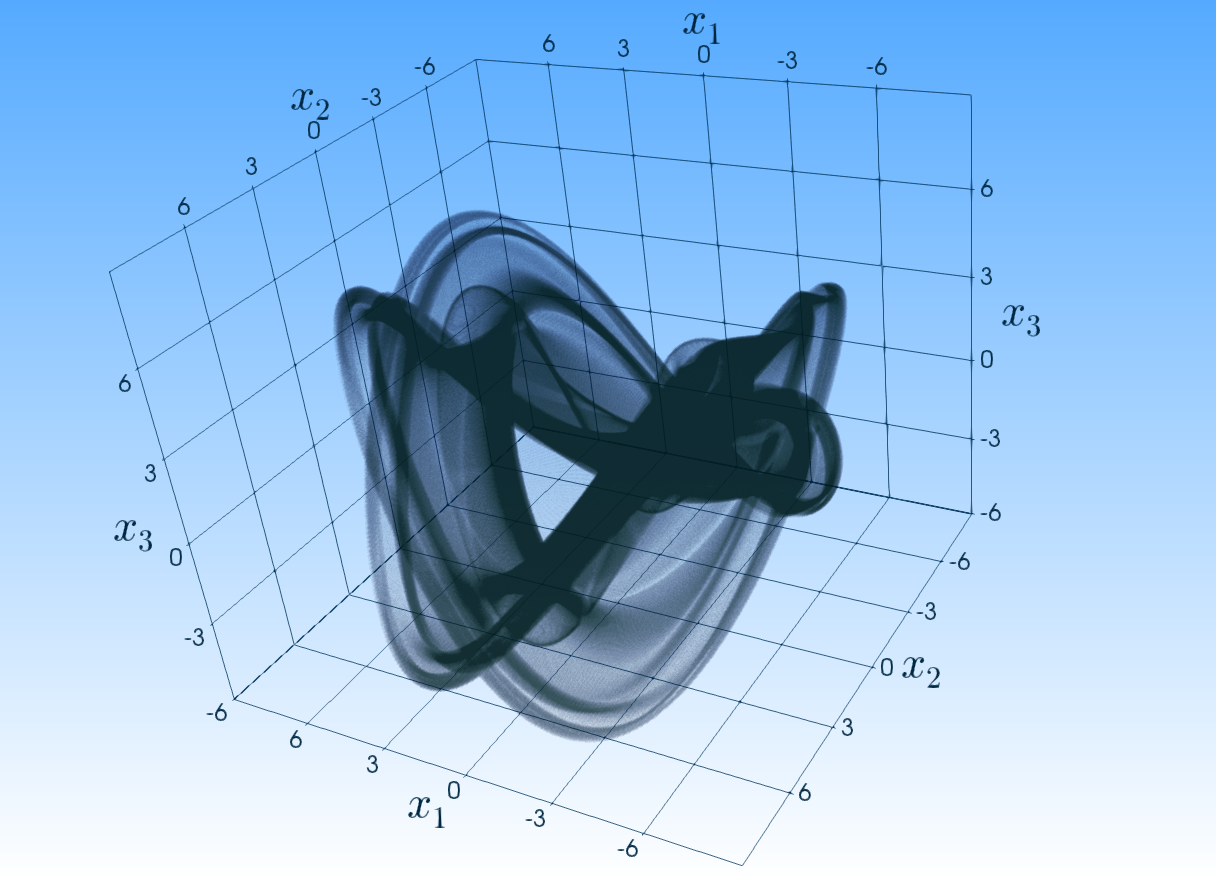}\\
		\centering \scriptsize{(d) $s = 63$ and $\ell = 0$}
	\end{minipage}
	\caption{(a)-(d) Successively finer box-coverings of the unstable manifold for $\mu = 15$. (d)~Transparent box covering for $s = 63$ and $\ell = 0$ depicting the internal structure of the unstable manifold.}
	\label{fig:KS_15}
\end{figure}

\subsubsection{The stable heteroclinic cycle}\quad\\
For $\mu = 18$, the observed long-term behavior consists of a pulsation between two states, which appear to be $\pi/2$-translations of each other. The transients linger close to one of these states for a comparatively long time before they pulse back to the other (cf. \cref{fig:KS18_simulation}~(a)). 

\begin{figure}[t!]
	\parbox[b]{0.49\textwidth}{\centering \includegraphics[width=0.49\textwidth]{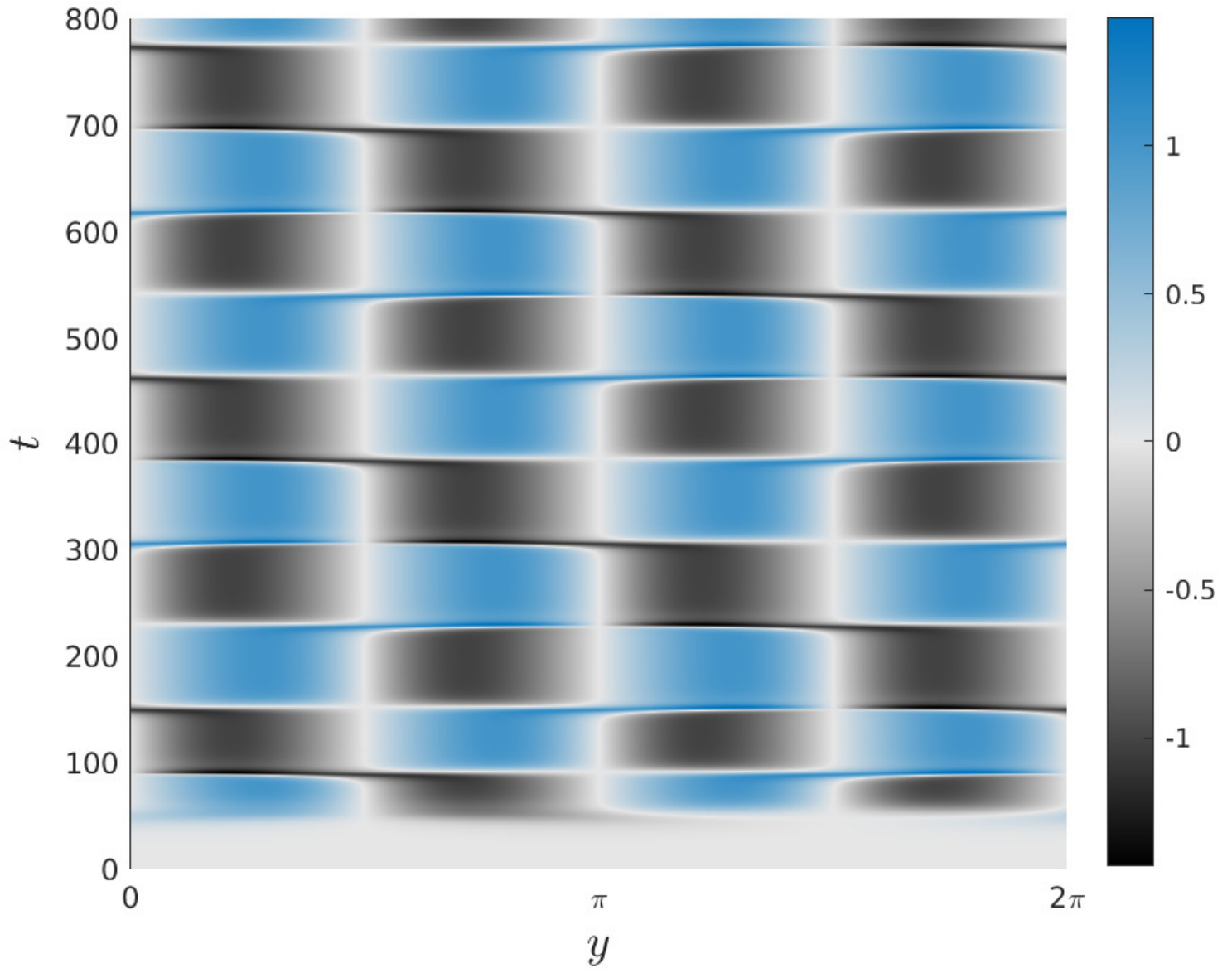}\\\scriptsize (a) Direct simulation}
	\parbox[b]{0.49\textwidth}{\centering \includegraphics[width=0.49\textwidth]{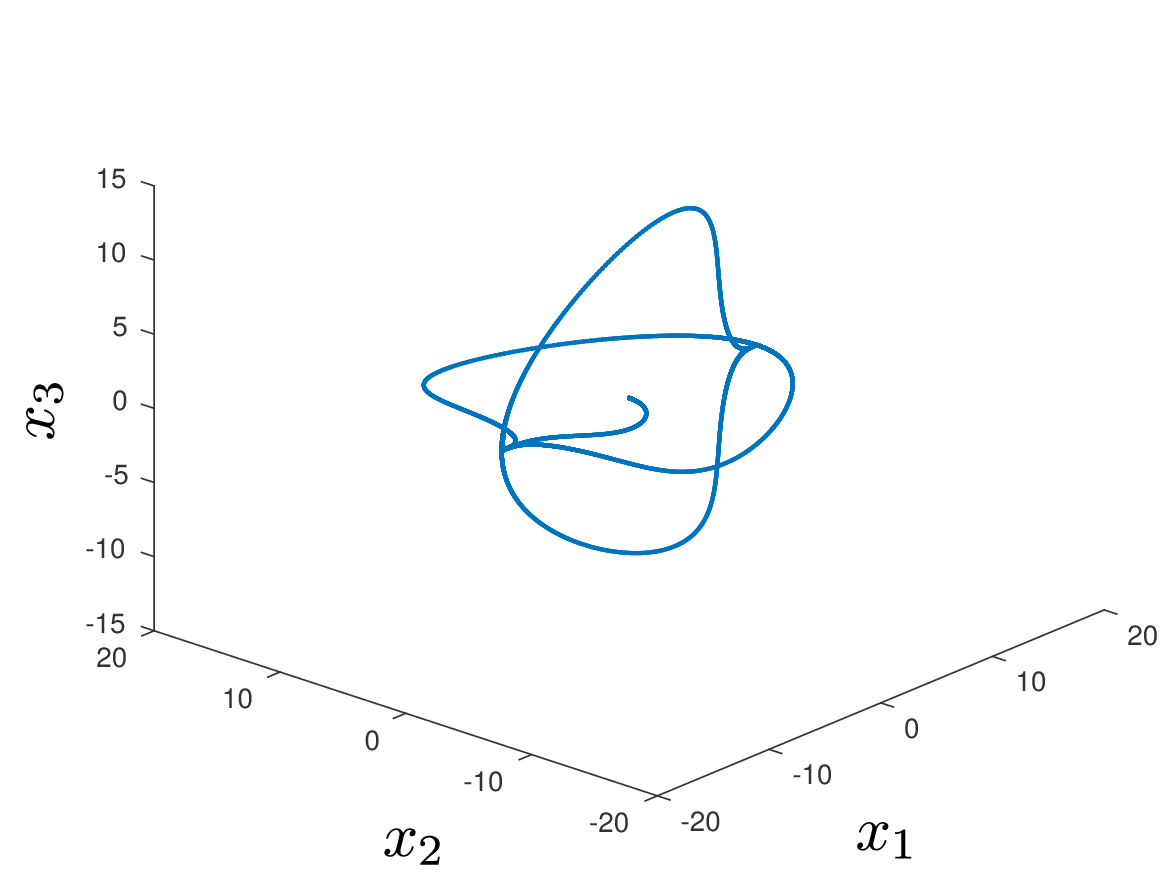}\\\scriptsize (b) Observation space}
	\caption{(a) Direct simulation of the Kuramoto-Sivashinsky equation for $\mu = 18$;\ (b) Corresponding embedding in observation space depicting the stable heteroclinic loop.}
	\label{fig:KS18_simulation}
\end{figure}

\begin{figure}[t!]
	\begin{minipage}{0.49\textwidth}
		\includegraphics[width = \textwidth]{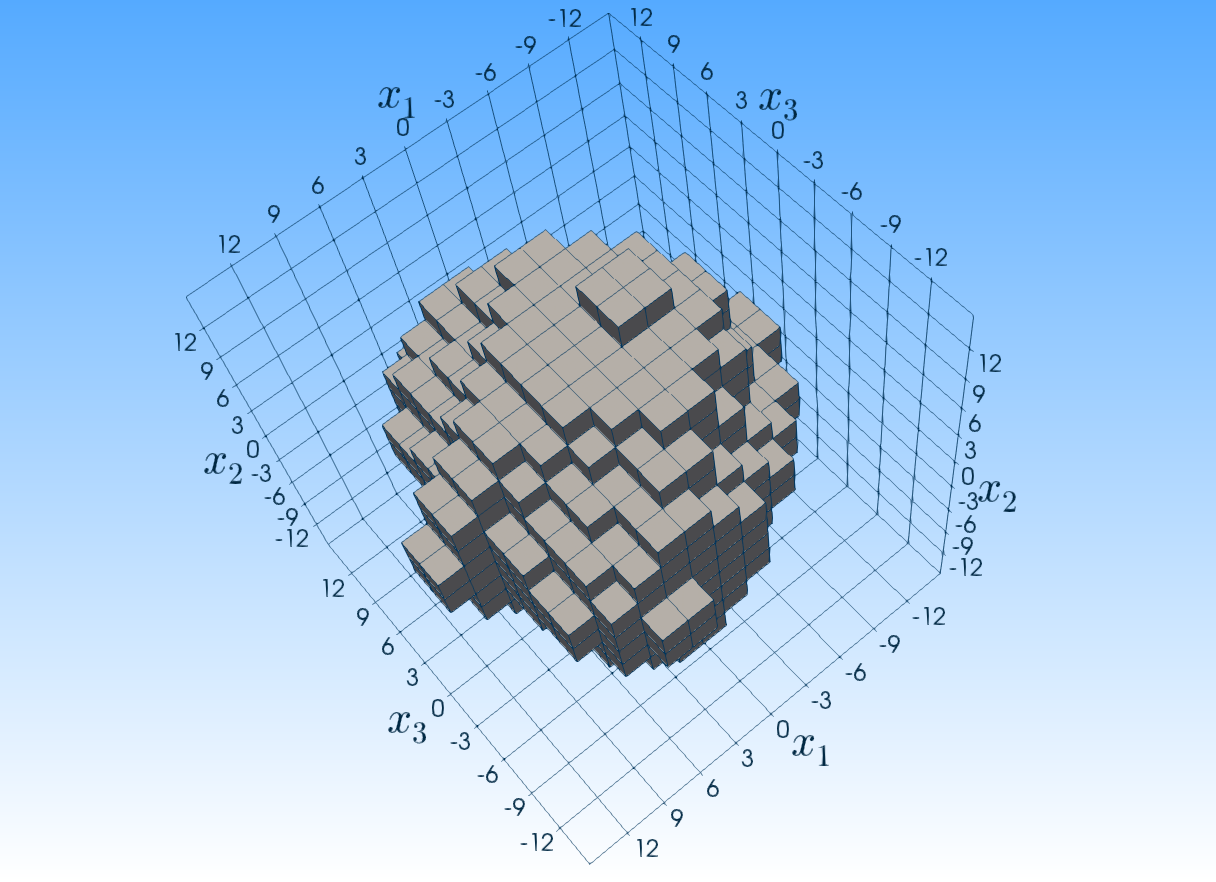}\\
		\centering \scriptsize{$s = 12$ and $\ell = 0$}
	\end{minipage}
	\hfill
	\begin{minipage}{0.49\textwidth}
		\includegraphics[width = \textwidth]{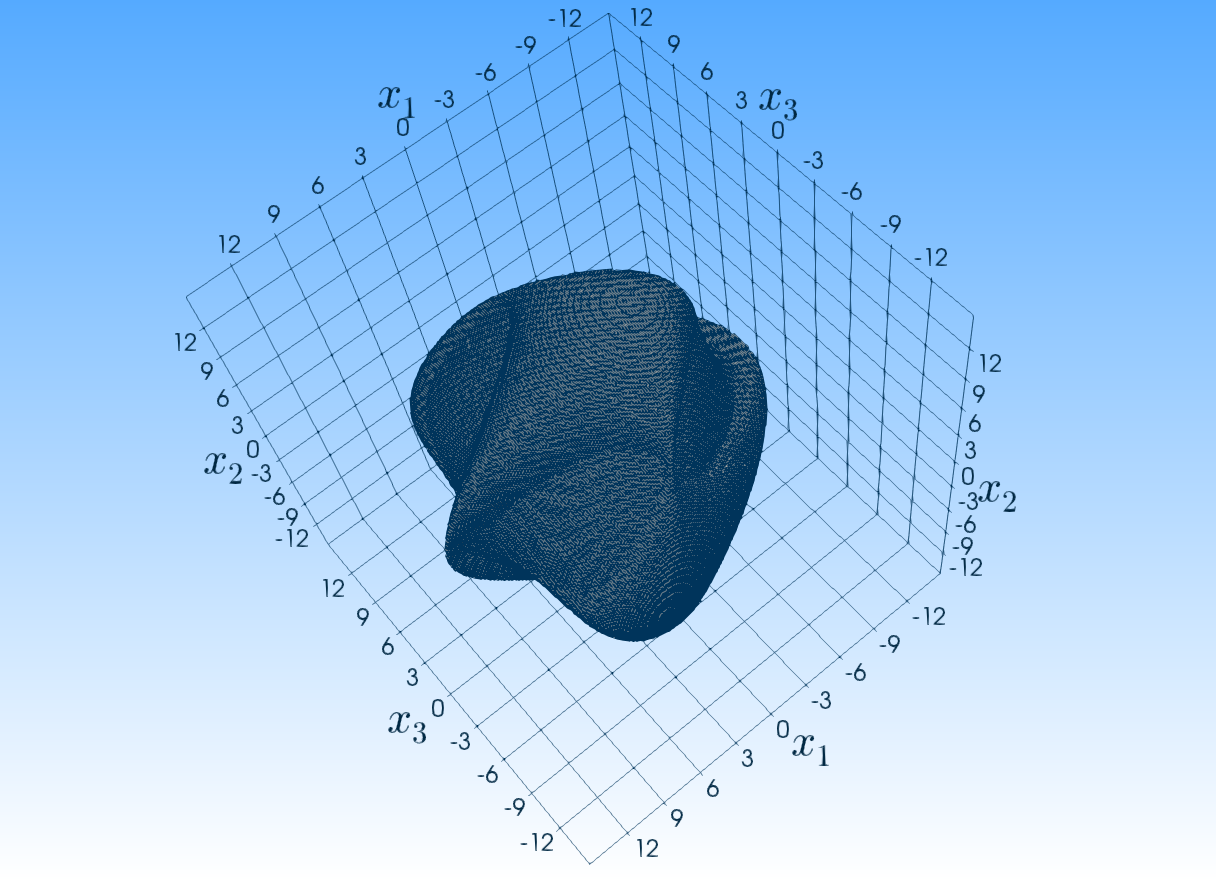}\\
		\centering \scriptsize{$s = 24$ and $\ell = 0$}
	\end{minipage}
	\caption{(a)-(b) Successively finer box-coverings of the unstable manifold for $\mu = 18$.}
	\label{fig:KS_18}
\end{figure}

It was observed in \cite{KNS90} that the pulsation projected onto the $\cos(2x)$ and $\sin(2x)$ coefficient plane, respectively, appears as a straight line passing through the origin. In addition, different pulsations, resulting from different initial conditions, give straight lines that are rotations of each other about the origin. By projecting the pulsation onto the first three POD-coefficients, we observe a similar behavior in observation space. Thus, we expect that the unstable manifold will be of dimension at least three. The projection of the long time simulation (cf.~\cref{fig:KS_18}~(a)) onto the first three POD-coefficients is shown in \cref{fig:KS18_simulation}~(b).

For the initialization of \cref{alg:continuation}, we choose the embedding dimension $k = 3$ and, therefore, restrict our initial functions to the function space generated by the first three POD-modes. Since the embedding dimension is too small, we expect to approximate just a projection of the unstable manifold. For a related discussion in the finite dimensional context we refer the interested reader to \cite{SYC91}. Moreover, we choose $Q = [-20,20]^3$ and set $T = 200$. In \cref{fig:KS_18}~(a) and (b) we show two box coverings obtained by our continuation method for different values $s \in \N$ of the partition $\cP_s$ of $Q$. As expected, we observe that the embedded unstable manifold appears to be a solid three-dimensional object. This corresponds to the observation mentioned above.

\begin{figure}[!htb]
	\parbox[b]{0.49\textwidth}{\centering \includegraphics[width=0.49\textwidth]{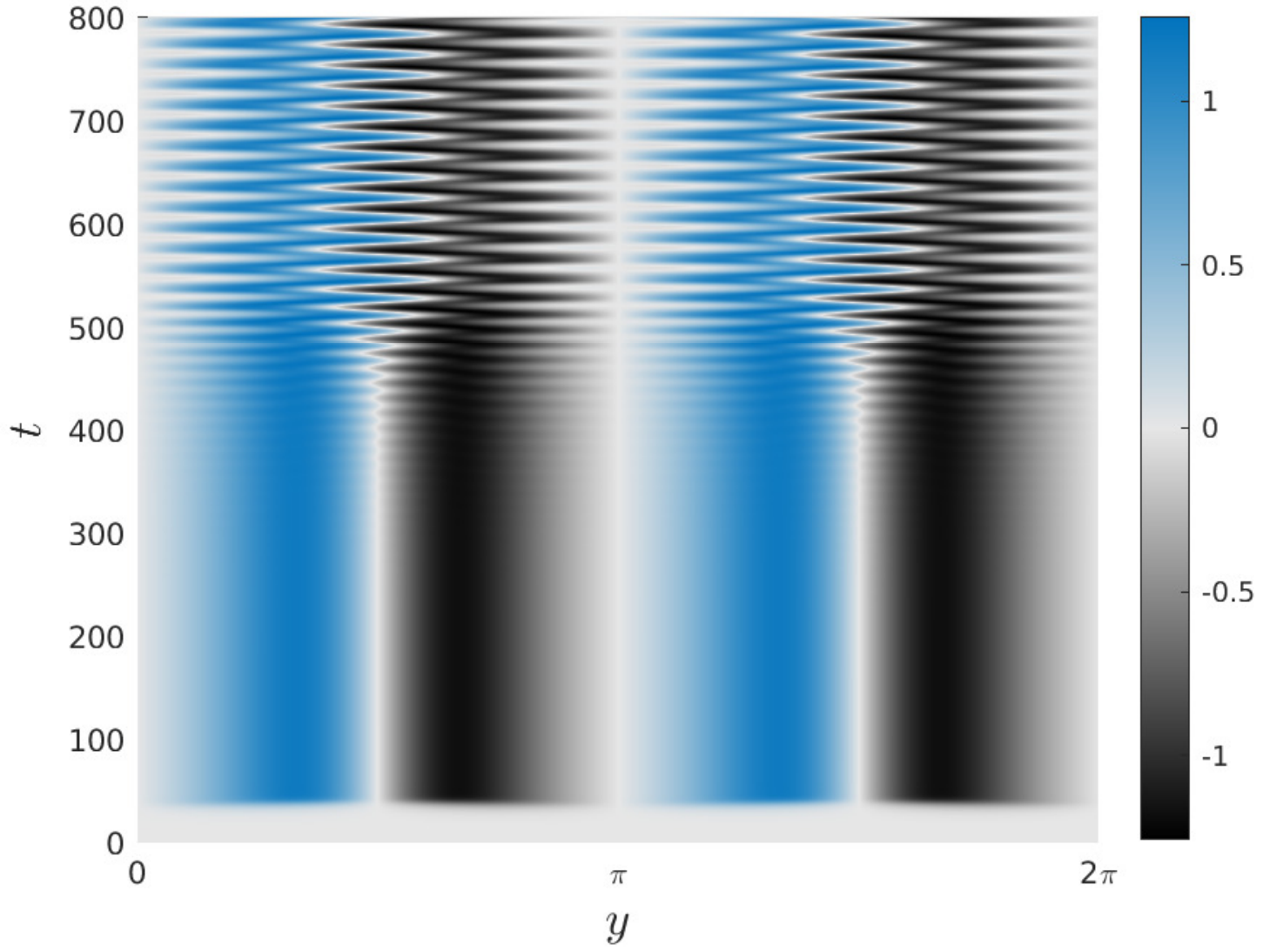}\\\scriptsize (a) Direct simulation}
	\parbox[b]{0.49\textwidth}{\centering \includegraphics[width=0.49\textwidth]{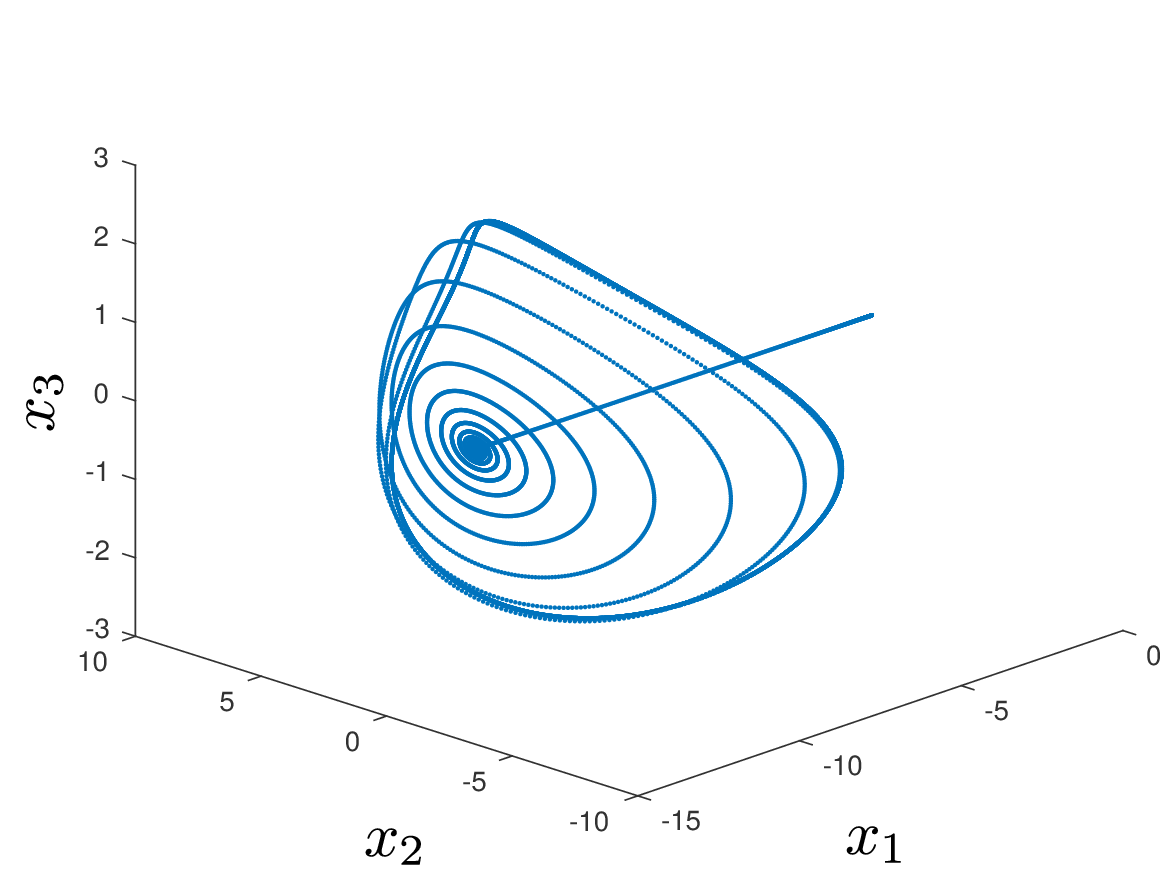}\\\scriptsize (b) Observation space}
	\caption{(a) Direct simulation of the Kuramoto-Sivashinsky equation for $\mu = 32$;\ (b) Corresponding embedding in observation space.}
	\label{fig:KS32_simulation}
\end{figure}
\subsubsection{The Oseberg transition}\label{sssec:oseberg}\quad\\
In the last example we have chosen $\mu=32$. In \cref{fig:KS32_simulation}~(a) and (b) we show a direct simulation 
as well as the corresponding embedding in the observation space. The initial condition $u_0$ is first attracted to an unstable so-called bimodal steady state, and eventually accumulates on a limit cycle as $t \rightarrow \infty$.

In \cref{fig:KS_32}~(a) and (b) we show successively finer box coverings of the unstable manifold obtained by our continuation method. We remark that already previously a similar result has been obtained by \cite{JJK01} which the authors called \emph{the Oseberg transition}. In this work, the authors restricted the phase space to the invariant subspace of odd functions, in which the solutions can be represented by the Fourier series
\[
u(y,t) = \sum_{j=1}^{\infty} b_j(t)\sin(jy),
\]
where an eight-mode Galerkin truncation of the PDE was used. The restricted global attractor, illustrated through the first, second and third Fourier coefficients as observables, look qualitatively very similar to our unstable manifold illustrated in \cref{fig:KS_32}.

\begin{figure}[!htb]
	\begin{minipage}{0.49\textwidth}
		\includegraphics[width = \textwidth]{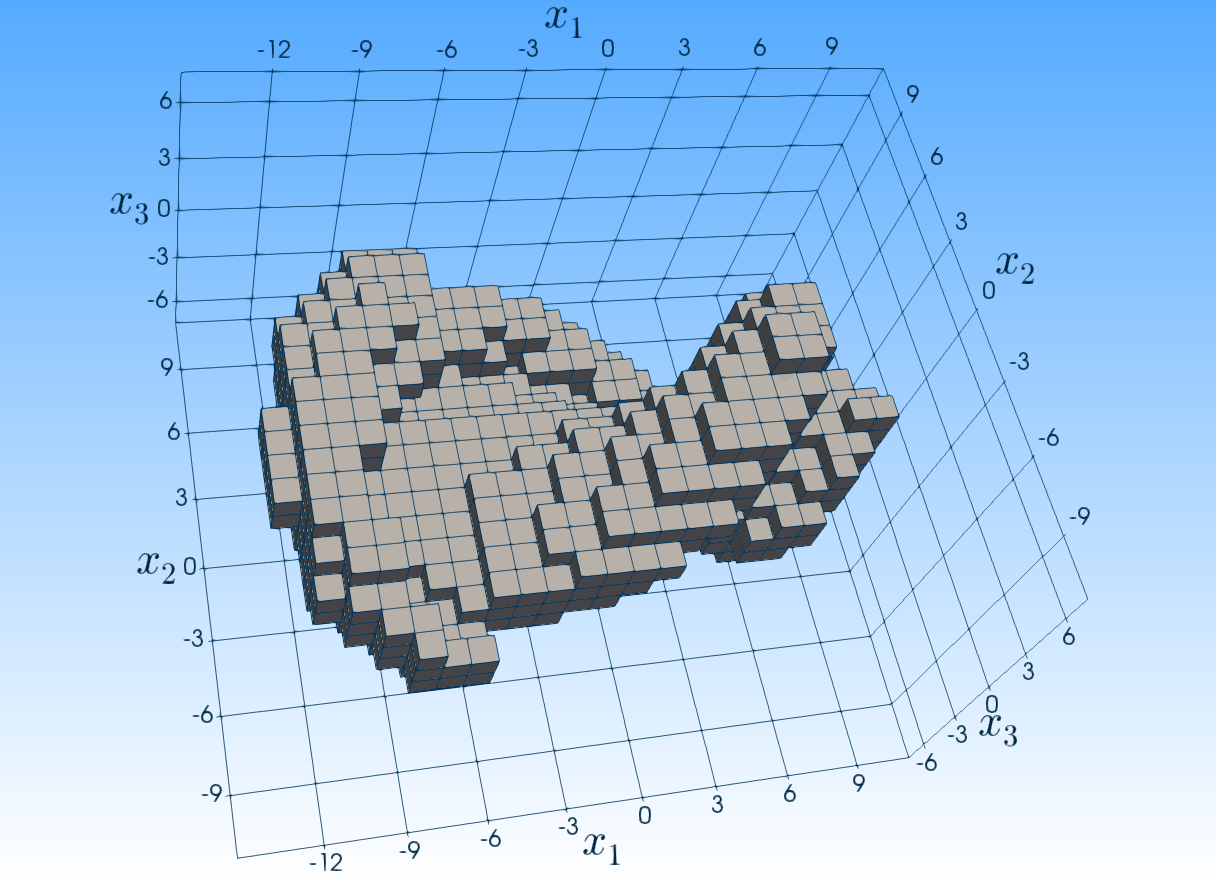}\\
		\centering \scriptsize{(a) $s = 15$ and $\ell = 0$}
	\end{minipage}
	\hfill
	\begin{minipage}{0.49\textwidth}
		\includegraphics[width = \textwidth]{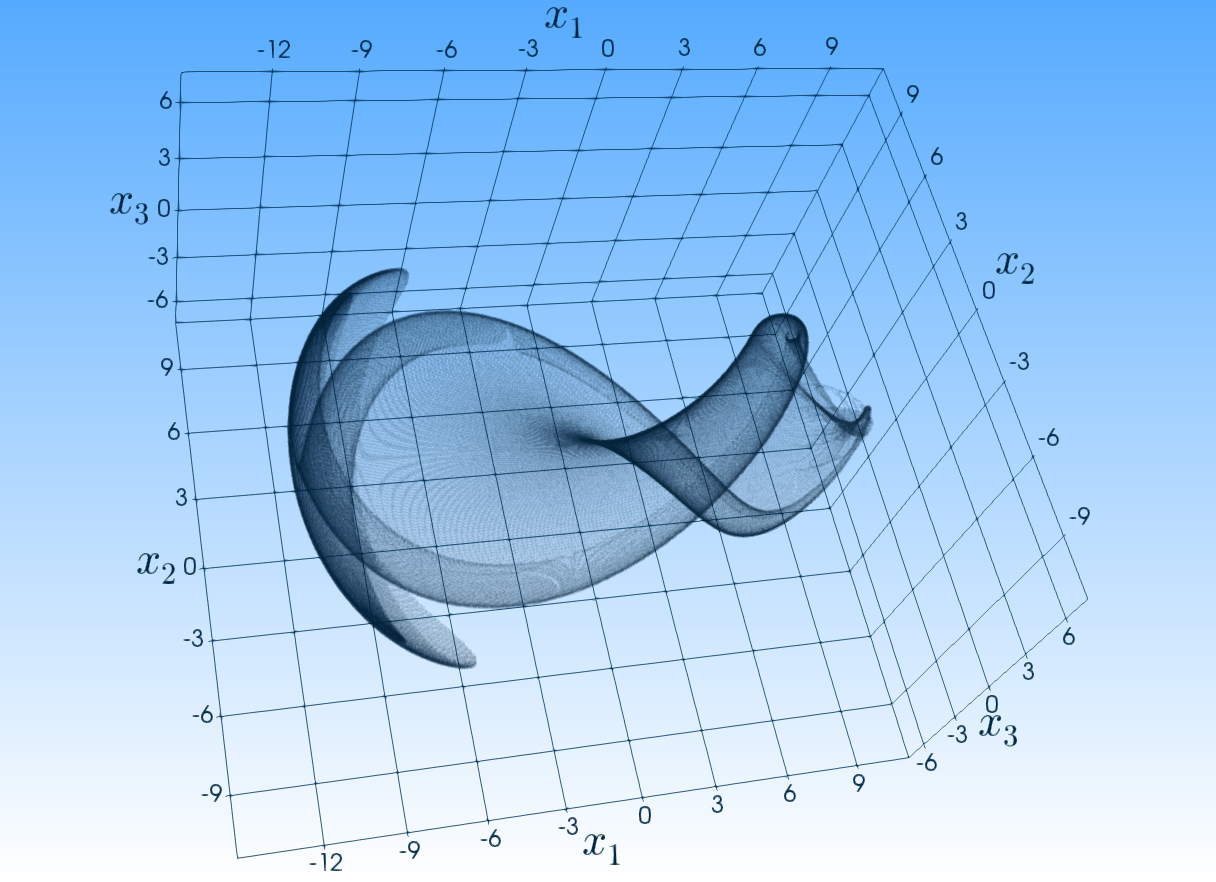}\\
		\centering \scriptsize{(b) $s = 27$ and $\ell = 0$}
	\end{minipage}
	\caption{Successively finer box-coverings of the unstable manifold for $\mu=32$. Furthermore, in (b) we show a transparent box covering for $s = 27$ and $\ell = 0$ depicting the internal structure of the unstable manifold.}
	\label{fig:KS_32}
\end{figure}

\begin{figure}[!htb]
	\begin{minipage}{0.49\textwidth}
		\includegraphics[width = \textwidth]{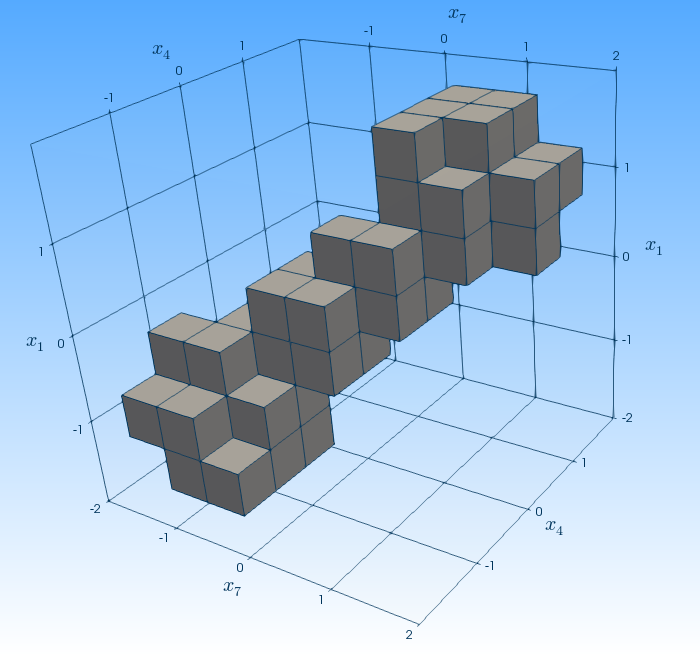}\\
		\centering \scriptsize{(a) $s = 21$ and $\ell = 0$}
	\end{minipage}
	\hfill
	\begin{minipage}{0.49\textwidth}
		\includegraphics[width = \textwidth]{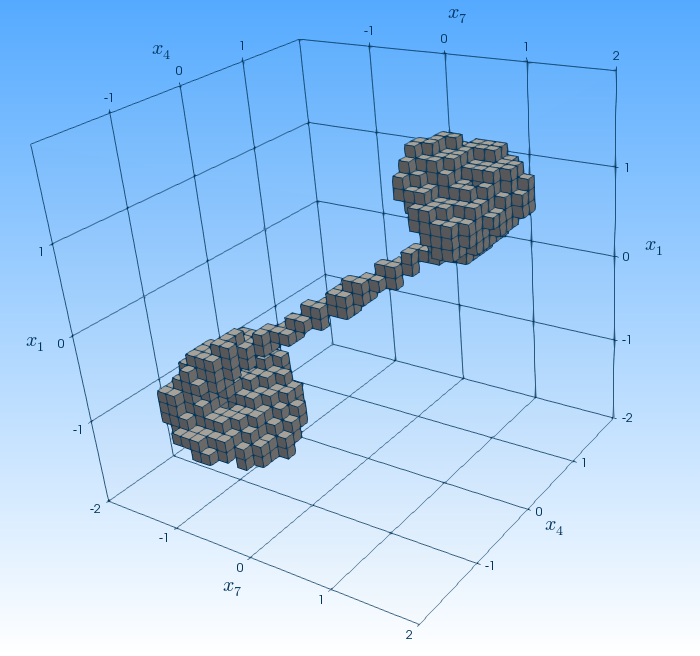}\\
		\centering \scriptsize{(b) $s = 35$ and $\ell = 0$}
	\end{minipage}\\[1em]
	\begin{minipage}{0.49\textwidth}
		\includegraphics[width = \textwidth]{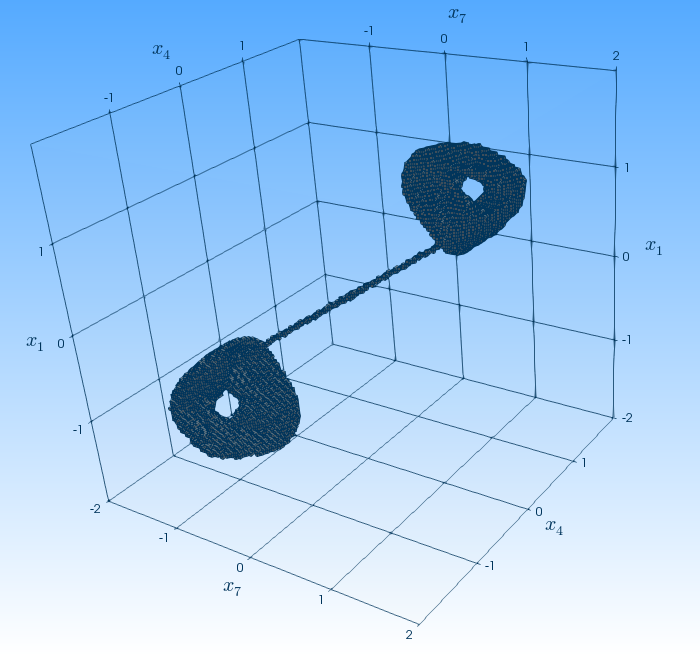}\\
		\centering \scriptsize{(c) $s = 49$ and $\ell = 0$}
	\end{minipage}
	\hfill
	\begin{minipage}{0.49\textwidth}
		\includegraphics[width = \textwidth]{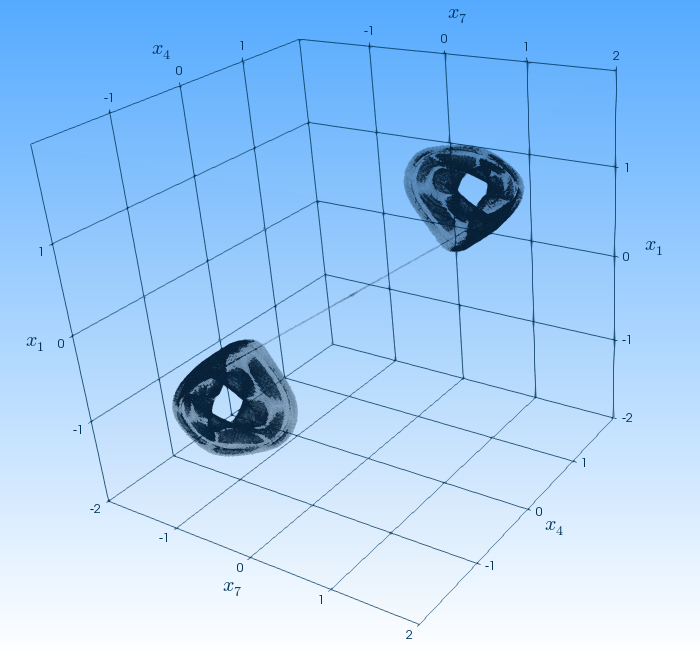}\\
		\centering \scriptsize{(d) $s = 49$ and $\ell = 14$}
	\end{minipage}
	\caption{(a)-(c) Successively finer box-coverings of the unstable manifold of the Mackey-Glass equation \eqref{eq:mg}. (d)~Transparent box covering for $s = 49$ and $\ell = 14$.}
	\label{fig:mg_continuation}
\end{figure}

\subsection{The Mackey-Glass equation}
\label{ssec:MG}
Finally, we show the feasibility of our method by computing an unstable manifold for one delay differential equation with constant delay. Here, we consider the delay differential equation introduced by Mackey and Glass in 1977 \cite{mackey1977} defined by
\begin{equation}\label{eq:mg}
\dot u(t) = \beta \frac{u(t-\tau)}{1+u(t-\tau)^\eta}-\gamma u(t),
\end{equation}
where we choose $\beta = 2, \gamma = 1$, $\eta = 9.65$, and $\tau = 2$. This equation is a model of blood production, where $u(t)$ represents the concentration of blood at time $t$, $\dot u(t)$ represents production at time $t$ and $u(t-\tau)$ is the concentration at an earlier time, when the request for more blood is made. This equation possesses an unstable steady state $u_0(t) = 0$. To this end, we start the continuation method in ${p = R(u_0) \in \R^k}$, where $R$ is defined by \eqref{eq:R_delay}. Furthermore, we choose $k = 7$ and $Q = [-1.5,1.5]^7 \subset \R^7$.

In \cref{fig:mg_continuation}~(a) -- (c), we show projections of the coverings obtained via the continuation method for $s = 21, 35, 49$ ($\ell = 0$ in all three cases). In addition, in \cref{fig:mg_continuation}~(d) we also show a transparent box covering for $s = 49$ and $\ell = 14$.

\section*{Acknowledgments} This work is supported by the Priority Programme SPP 1881 Turbulent Superstructures of the Deutsche Forschungsgemeinschaft.

\bibliographystyle{plain}
\bibliography{references}

\end{document}